\documentclass[11pt,a4paper,final]{amsart}

\usepackage{amsmath}
\usepackage{paralist}
\usepackage{graphics}
\usepackage{graphicx}
\usepackage{epstopdf}
\usepackage{amssymb}
\usepackage{amsthm}
\usepackage{color}
\usepackage{multirow}
\usepackage{booktabs}
\usepackage{bm}
\usepackage{mathrsfs}

\usepackage[normalem]{ulem}

\newtheorem{theorem}{Theorem}[section]
\newtheorem{corollary}{Corollary}

\newtheorem{remark}{Remark}

\newcommand{\mc}[1]{{\mathcal #1}}
\newcommand{\mb}[1]{{\mathbf #1}}

\begin{document}

\title[Diffusion and transport on curves]{On diffusion and transport acting on parameterized moving closed curves in space}

\author{Michal Bene\v{s}${}^{1}$}
\author{Miroslav Kol\'a\v{r}${}^{1}$}
\address{${}^{1}$ Department of Mathematics, Faculty of Nuclear Sciences and Physical Engineering Czech Technical University in Prague, Trojanova 13, Prague, 12000, Czech Republic}
\author{Daniel \v{S}ev\v{c}ovi\v{c}${}^{2}$}
\address{${}^{2}$ Department of Applied Mathematics and Statistics, Faculty of Mathematics Physics and Informatics, Comenius University, Mlynsk\'a dolina, 842 48, Bratislava, Slovakia. Corresponding author: {\tt sevcovic@fmph.uniba.sk} }

\begin{abstract}
We investigate the motion of closed smooth curves that evolve in space $\mathbb{R}^3$. The governing evolutionary equation for the evolution of the curve is accompanied by a parabolic equation for the scalar quantity evaluated over the evolving curve. 
We apply the direct Lagrangian approach to describe the flow of 3D curves, resulting in a system of degenerate parabolic equations.
We prove the local existence and uniqueness of classical H\"older smooth solutions to the governing system of nonlinear parabolic equations.  A numerical discretization scheme is constructed using the method of flowing finite volumes. We present several numerical examples of the evolution of curves in 3D with a scalar quantity. We consider the flow of curves with zero torsion evolving in rotating and parallel planes. Next, we present examples of the evolution of curves with initially knotted and unknotted curves.

\medskip
\noindent
2010 MSC. Primary: 35K57, 35K65, 65N40, 65M08; Secondary: 53C80.

\noindent Key words and phrases. Curvature and binormal driven flow, non-local flow, Biot-Savart law, analytic semi-flows, H\"older smooth solutions, flowing finite volume method, knotted curves evolution. 

\end{abstract}

\maketitle

\section{Introduction}

The problem of distribution of scalar quantities along moving curves and interfaces arises in many real applications in nature and technology. In fluid dynamics, vortex structures can be formed along a one-dimensional open or often closed curve that represents a vortex ring. We refer to Meleshko et al.  \cite{Meleshko:12} for a review, and Fukumoto et al. \cite{Fukumoto1987, Fukumoto1991} for examples. These structures, in fact, have a finite cross section varying along the vortex and can contain another phase, e.g. air in water. The difference in densities causes the cross section to be greater in the upper parts of the vortex ring than in the lower parts. The vortex cross-section or diameter becomes a scalar quantity to be transported along the vortex curve (see also \cite{Padilla19}). Certain defects in the crystalline lattice as linear structures, called dislocations, form macroscopic material properties (see Hirth, Lothe \cite{Hirth}, or Kubin \cite{Kubin}). Their specific configurations (jogs) can be described by scalar quantities displaced along the dislocations (see Niu \emph{et al.} \cite{niu2019,niu2017}). Nanofibers can be produced by electrospinning - jetting polymer solutions in high electric fields into ultrafine nanofibers (see Reneker \cite{Reneker}, Yarin \emph{et al.} \cite{Yarin}, He \emph{et al.} \cite{He}). These structures move freely in space and can dry out during electrospinning from a solution - exhibit internal mass transfer (see \cite{Wu}). A closed simple curve can also represent a 2D cut through the myocardium surface, the heart muscle operated by an electric signal broadcast by the sinoatrial node \cite{Guy06}. This signal is represented by ions located along the cell membrane, their potential together with the ability of cell membranes to allow the motion of ions \cite{Kee98} are the scalar variables appearing in the FitzHugh-Nagumo reaction-diffusion model \cite{KanBen21} on the curved myocardium-surface curve.
Thin liquid pathways occur between grains and ice during soil freezing and mediate mass transport responsible for mechanical effects such as frost heave (see, e.g. \cite{Rempel:04, Suh:22, Zak:23}). Following the above motivating applications, we consider the following advection-diffusion problem along moving curves
$\Gamma_t, t\ge0$:
\begin{equation}
\mb{v} = a \kappa \mb{N} + b \kappa \mb{B} +  \mb{F}, 
\qquad 
\partial_t \varrho = c \Delta_{\Gamma} \varrho + G, 
\label{Tr:1}
\end{equation}
where $\mb{v}$ is the velocity vector of the curve, $\varrho$ is the scalar quantity displaced along the curve $\Gamma_t$.
The triple $\mb{N}$, $\mb{T}$, $\mb{B}$ is the Frenet frame composed of normal, tangent, and binormal vectors to $\Gamma_t$, respectively. In the first equation, $a$ and $b$ are functions influencing motion in the normal and binormal directions and $\mb{F}$ is a general force acting on the curve $\Gamma_t$.  
The force $\mb{F}$ may contain the tangential velocity part $\alpha=\mb{F}\cdot\mb{T}$, as well as non-local quantities that can be expressed as curve integrals of $\mb{X}$ along evolved curve. The general form (\ref{eq:F})) of a non-local force includes, e.g., the Biot-Savart law (\ref{biot-savart-force}) described in Section~\ref{sec-knotted} (see, Fig.~\ref{fig-fourier}). In the second equation
$G$ is the source term, $c>0$ is the diffusion coefficient. The source term $G$ may contain tangential redistribution effects induced by the dynamics of $\Gamma_t$. Here $\Delta_{\Gamma}$ is the Laplace-Beltrami operator with respect to a one-dimensional curve $\Gamma_t$, that is, $\Delta_{\Gamma} = \partial^2_s$ where $s$ is the arc-length parameter.

In this paper, we extend the mathematical results obtained in \cite{BKS2022} for the motion law of the curve to the general transport problem (\ref{Tr:1}). We investigate the motion of closed smooth curves that evolve in space $\mathbb{R}^3$. The governing evolution equation for the the curve~$\Gamma_t$ is accompanied by a parabolic equation for the scalar quantity distributed along the curve~$\Gamma_t$. 
Theoretical analysis of the motion of space curves is contained, among others, in papers by Altschuler and Grayson in \cite{altschuler1991} and \cite{altschuler1992}. We follow the direct Lagrangian approach to describe a flow of 3D curves. The direct Lagrangian approach was applied for planar curve evolution by many authors, e.g. Deckelnick \cite{Deckelnick1995}, Mikula and \v{S}ev\v{c}ovi\v{c} \cite{sevcovic2001evolution,MS2004, MMAS2004},  and others. The resulting system of parabolic equations is coupled with a parabolic equation for the scalar quantity evaluated over the evolved curves. 

We prove the local existence, uniqueness of classical H\"older smooth solutions to the governing system of nonlinear parabolic equations. The main idea of the proof of existence and uniqueness of H\"older smooth solutions is based on the abstract theory of analytic semiflows in Banach spaces due to DaPrato and Grisvard \cite{daprato}, Angenent \cite{Angenent1990, Angenent1990b}, Lunardi \cite{Lunardi1984}.
The numerical approximation scheme is based on the flowing finite-volume method which was proposed by Mikula and \v{S}ev\v{c}ovi\v{c} in \cite{sevcovic2001evolution} for curvature-driven flows of planar curves. In the first numerical examples, we consider evolution of curves with no torsion evolving in rotating and parallel planes. In the second set of examples, we investigate the evolution of the initial knotted curves.

The structure of the paper is as follows. In the following section, we provide a review of the direct Lagrangian approach used to solve curvature-driven flows of a family of curves in three dimensions. In Section 3, we establish a system of non-local partial differential equations for the parameterizations of the evolving curves. Section 4 provides a brief summary of the importance of tangential velocity in solving analytical and numerical problems. The local existence and uniqueness of classical H\"older smooth solutions are demonstrated in Section 5. Section 6 presents semi-analytical examples of curve evolution in rotating and parallel planes. In Section 7, we develop a numerical discretization scheme that utilizes the method of flowing finite volumes. Finally, Section 8 presents numerical examples of evolving curves and scalar quantities.

\section{Direct Lagrangian description for evolution of closed curves}

In this paper, we consider a family $\{\Gamma_t, t\ge 0\}$ of curves that evolve in space $\mathbb{R}^3$. Its evolution can be described by the position vector
$\mb{X} = \mb{X}(t,u)$ for $t \ge 0$ as follows:
\[
\Gamma_t = \{ \mb{X}(t,u),  u \in I \},
\]
where $I=\mathbb{R}/\mathbb{Z}\simeq S^1$ denotes the periodic interval. Therefore, we will identify $I$ with the interval $[0,1]$. The unit tangent vector $\mb{T}$ to $\Gamma_t$ is defined as $\mb{T} = \partial_s \mb{X}$, where $s$ is the unit arc-length parameterization defined by $ds = |\partial_u \mb{X}| du$. Here, $|\cdot|$ denotes the Euclidean norm. The curvature $\kappa$ of a curve $\Gamma$ is defined as $\kappa = |\partial_{s}\mb{X}\times\partial^2_{s} \mb{X} | = | \partial_{s}^2 \mb{X} |$. If $\kappa > 0$, we can define the Frenet frame along the curve $\Gamma_t$ with unit normal $\mb{N} = {\kappa}^{-1} \partial_{s}^2 \mb{X}$ and binormal vectors $\mb{B} = \mb{T} \times \mb{N}$, respectively. Recall the Frenet-Serret formulae: 
\[
\frac{\text{d}}{\text{d} s} 
\begin{pmatrix}
\mb{T} \\ \mb{N} \\ \mb{B}
\end{pmatrix}
=
\begin{pmatrix}
0 & \kappa & 0 \\
-\kappa & 0 & \tau \\
0 & -\tau & 0
\end{pmatrix}
\begin{pmatrix}
\mb{T} \\ \mb{N} \\ \mb{B}
\end{pmatrix},
\]
where $\tau$ is the torsion of $\Gamma_t$ given by
$\tau =
\kappa^{-2} (\mb{T}\times\partial_s\mb{T}) \cdot \partial_s^2\mb{T}
= \kappa^{-2} (\partial_s\mb{X}\times\partial_s^2\mb{X}) \cdot \partial_s^3\mb{X}$. 
We study a coupled system of evolutionary equations describing evolution of closed 3D curves evolving in normal and binormal directions, and the scalar quantity $\varrho$ computed over the evolving family of curves,
\begin{equation}
\label{eq:ab}
\begin{split}
 \partial_t\mb{X} &= a \partial^2_{s} \mb{X} + b (\partial_{s}\mb{X}\times\partial^2_{s} \mb{X})  + \mb{F}(\mb{X}, \partial_{s}\mb{X} ,\varrho, \partial_s\varrho),\\
 \partial_t\varrho &= c \partial^2_{s} \varrho + G(\mb{X}, \partial_{s}\mb{X}, \varrho, \partial_s\varrho,  \kappa\beta),
\end{split}
\end{equation}
where the following scalar functions $a=a(\mb{X}, \partial_{s}\mb{X}, \varrho, \partial_s\varrho) > 0$, $b=b(\mb{X}, \partial_{s}\mb{X}, \varrho, \partial_s\varrho)$, and $c=c(\mb{X}, \partial_{s}\mb{X}, \varrho, \partial_s\varrho)>0$, and $\beta$ is the normal velocity, $\beta  = (a \partial^2_{s} \mb{X} + \mb{F}) \cdot \mb{N}$. The source terms $\mb{F}$ and $G$ are assumed to be bounded and smooth functions of their arguments.
Both source terms $\mb{F}$ and $G$ for the curve $\Gamma_t$ and for the scalar quantity $\varrho$, respectively, can depend locally on the position vector $\mb{x}$, the tangent vector $\mb{t}$, the scalar quantity $\varrho$, and non-locally on the entire evolving curve $\Gamma_t$ (e.g., on the length of $\Gamma_t$ or on an another integral quantity over $\Gamma_t$). We assume that the source term $G$ may depend on the curvature $\kappa=|\partial_{s}\mb{X}\times\partial^2_{s} \mb{X} | = | \partial_{s}^2 \mb{X} |$. Thus, the parabolic equation for $\varrho$ may depend on the highest-order derivative $| \partial_{s}^2 \mb{X}|$. As an example, one can consider the case where 
\begin{equation}
\mb{F}(\mb{x}, \mb{t}) = \int_{\Gamma} f(\mb{x}, \mb{t}, \mb{X}, \mb{\partial_s\mb{X}}) ds
= \int_0^1 f(\mb{x}, \mb{t}, \mb{X}(u), \partial_u\mb{X}(u)/ | \partial_u\mb{X}(u)|) |\partial_u \mb{X}| du. 
\label{eq:F}
\end{equation}
Here $f:\mathbb{R}^3\times \mathbb{R}^3 \times \mathbb{R}^3\times \mathbb{R}^3\to \mathbb{R}^3$,  is a given smooth function.

Since $\partial^2_s \mb{X} =\kappa \mb{N}$ and $\mb{B}= \mb{T} \times \mb{N}$ the governing equation of the system (\ref{eq:ab}) for the position vector $\mb{X}$ and the evolution of $\Gamma_t$ can be expressed in terms of the governing equation
\begin{equation}
    \label{eq:mcfl}
    \partial_t \mb{X} = \beta \mb{N}  + \gamma \mb{B} + \alpha \mb{T}, 
    \qquad \beta = a \kappa + \mb{F} \cdot \mb{N}, \quad
\gamma = b \kappa + \mb{F} \cdot \mb{B}, \quad
\alpha = \mb{F} \cdot \mb{T}, 
\end{equation}
where $\beta$, $\gamma$ and $\alpha$ are normal, binormal, and tangential components of the velocity, respectively.  Denote by $L(\Gamma_t)$ the total length of the curve $\Gamma_t$ parameterized by $\mb{X}$:
\[
L(\Gamma_t)= \int_{\Gamma}  ds= \int_0^1 |\partial_u \mb{X}| du. 
\]
Furthermore, if the curve $\Gamma_t$ evolves according to the governing equation (\ref{eq:mcfl}), we can derive an equation for the time derivative of the local length $g(t,u)|\partial_u \mb{X}(t,u)|$ of the curve $\Gamma_t$:
\begin{equation}
    \label{eq:dtg}
    \partial_t|\partial_u \mb{X}| = 
    |\partial_u \mb{X}| (\partial_s \alpha - \beta \kappa), \qquad \partial_t L(\Gamma_t) = - \int_{\Gamma} \beta \kappa ds, 
\end{equation}
where $\alpha$ and $\beta$ are tangential and normal components of the velocity in (\ref{eq:mcfl}), respectively.

\section{Parabolic equation for a scalar quantity on an evolving curve}

Assume that $\hat\varrho$ is a time-dependent function $\hat\varrho(t, \cdot): \Gamma_t\to\mathbb{R}$, $t\in[0,T]$, where $\Gamma_t$ is a curve parameterized by its position vector $\mb{X}$. 
In what follows, we shall denote by $\phi(t,u)$ the scalar quantity evaluated at time $t$ and the parameter $u\in[0,1]$ while $\hat\phi(t,\mb{X})$ stands for a scalar quantity evaluated at time $t$ and a point $\mb{X}\in\Gamma_t$ related to each other through equality $\phi(t,u):=\hat\phi(t,\mb{X}(t,u))$. Let $t\in [0,T]$ be fixed. Consider the arc-length integral $m(t) = \int_{\widehat{\mb{X}_1\mb{X}_2}} \hat\varrho(t,\mb{X}) \text{d}s$  of the scalar quantity $\hat \varrho$ along the curve segment $\widehat{\mb{X}_1\mb{X}_2}\subset \Gamma_t$ between the points $\mb{X}_1, \mb{X}_2\in \Gamma_t$. Notice that $m(t)$ is an intrinsic quantity, that is, it is independent of a particular parameterization of the curve $\Gamma_t$. If we introduce a parametrization dependent quantity $\varrho(t,u) =\hat\varrho(t, \mb{X}(t,u))$ then  
\begin{equation}
m(t)= \int_{u_1}^{u_2} \varrho(t,u) |\partial_u \mb{X}(t,u)| \text{d}u,
\label{concentrationdef}
\end{equation}
where $\mb{X}_i=\mb{X}(t,u_i)\in \Gamma_t, i=1,2$.
Hence $\varrho(t,u) =\hat\varrho(t, \mb{X}(t,u))$ is the density along the curve $\Gamma_t$ evaluated in terms of the parameter $u \in [u_1,u_2]\subseteq [0,1], \mb{X}_i=\mb{X}(t,u_i)$. 
Following standard textbooks by Slattery \cite{Slatt1}, \cite[Chapter 1.3]{Slatt2}, we can derive an advection-diffusion equation for the scalar quantity $\varrho$, we assume that the time derivative of $m(t)$ can be expressed in terms of inflow and outflow of flux $\hat j(t,\mb{X})$ through the end points $\mb{X}_1$ and $\mb{X_2}$ of the segment $\widehat{\mb{X_1}\mb{X}_2}$, and the prescribed external source term with a density $\hat{q}(t, \mb{X})$ of the segment as follows: 
\begin{eqnarray*}
\frac{\text{d}m}{\text{d}t} &=& 
- \left[\hat j(t,\mb{X})\right]_{\mb{X}_1}^{\mb{X}_2} + \int_{\widehat{\mb{X_1}\mb{X}_2}} \hat q(t, \mb{X}) \text{d} s  = - \left[j(t,u)\right]_{u_1}^{u_2} + \int_{u_1}^{u_2} q(t,u) |\partial_u \mb{X}| \text{d} u 
\\
&=& -\int_{u_1}^{u_2} \frac{\partial}{\partial u} j(t,u) \text{d} u 
+ \int_{u_1}^{u_2} q(t,u) |\partial_u \mb{X}| \text{d} u 
\\
&=& - \int_{u_1}^{u_2} \frac1{|\partial_u \mb{X}|} \frac{\partial}{\partial u} j(t,u) |\partial_u \mb{X}| \text{d} u + \int_{u_1}^{u_2} q(t,u) |\partial_u \mb{X}| \text{d} u
\end{eqnarray*}
where $j(t,u)=\hat j(t, \mb{X}(t,u)), q(t,u)=\hat q(t, \mb{X}(t,u))$. 

We assume that the mass flux along the curve is expressed as a sum of the advection term and the linear Fickian term with a diffusion constant $k>0$. The constitutive law is given by right-hand side of the following equation:
\[
\hat j  = - c \partial_s \hat\varrho + v \hat\varrho =  -c \frac{1}{|\partial_u \mb{X}|} \frac{\partial \varrho}{\partial u} + v \varrho, 
\]
where $\partial_s$ denotes the derivative with respect to the arc-length parameterization $s$, and $v(t,u)=\hat v(t, \mb{X}(t,u))$ denotes the advection velocity along the curve. By substituting, we obtain the following equation:
\begin{equation}
    \label{eq:balance_integral}
    \frac{\text{d}m}{\text{d}t} = 
    \int_{u_1}^{u_2} \left[ 
    \frac{1}{|\partial_u \mb{X}|} \frac{\partial}{\partial u} \left(
    \frac{c}{|\partial_u \mb{X}|} \frac{\partial \varrho}{\partial u}
    \right)
    - \frac{1}{|\partial_u \mb{X}|} \frac{\partial (v \varrho)}{\partial u}
    +q
    \right]
    |\partial_u \mb{X}| \text{d}u.
\end{equation}
Here, we can see that the tangential motion of $\Gamma_t$ can also influence the distribution of the quantity $\varrho$ along $\Gamma_t$. Using (\ref{eq:dtg}) in (\ref{eq:balance_integral}) we finally obtain the advection-diffusion equation: 
\begin{equation}
    \label{eq:balance_differential}
    \frac{\partial \varrho}{\partial t}  + 
    \varrho \left(  \partial_s \alpha - \kappa \beta \right) = 
    c \partial^2_s \varrho - \partial_s (\varrho v)  + q 
\end{equation}
for the scalar quantity $\varrho$ depending on the evolving curve $\Gamma_t$.

\section{The role of the tangential velocity}

Recall that the tangential component $\alpha$ of the velocity of evolving family of closed curves $\{\Gamma_t, t\ge 0\}$ has no impact on the shape of evolving curves (see, e.g., Epstein and Gage \cite{EpsteinGage}). On the other hand, considering a numerical solution of (\ref{eq:mcfl}), properly chosen tangential velocity functional $\alpha= \mb{F} \cdot \mb{T}$ plays an important role in the stability of a computational scheme (see, e.g., Mikula and \v{S}ev\v{c}ovi\v{c} \cite{sevcovic2001evolution, MS2004, MMAS2004}). The tangential velocity has a significant impact on the analysis of the evolution of curves from both the analytical and numerical points of view. It was shown by Hou et al. \cite{Hou}, Kimura \cite{Kimura}, Mikula and \v{S}ev\v{c}ovi\v{c} \cite{sevcovic2001evolution, MS2004, MMAS2004}, Yazaki and \v{S}ev\v{c}ovi\v{c} \cite{SevcovicYazaki2012}.  Barrett \emph{et al.} \cite{Barret2010, Barret2012}, Elliott and Fritz \cite{Elliot2017}, investigated the gradient and elastic flows for closed and open curves in $\mathbb{R}^d, d\ge 2$. They constructed a numerical approximation scheme using a suitable tangential redistribution. Bene\v{s}, Kol{\' a}{\v r} and \v{S}ev\v{c}ovi\v{c} investigated the role of tangential velocity in the context of material science \cite{BKS2017} and the evolution of interacting curves \cite{BKS2020}, \cite{BKS2022}. In \cite{Garcke2009} Garcke, Kohsaka and \v{S}ev\v{c}ovi\v{c} applied the uniform tangential redistribution in the theoretical proof of nonlinear stability of stationary solutions for curvature driven flow with triple junction in the plane. In \cite{MS2014} Reme\v{s}\'{\i}kov\'a \emph{et al.} proposed and analyzed the tangential redistribution for flows of closed manifolds in $\mathbb{R}^n$. Using equations (\ref{eq:dtg}) we can calculate the time derivative of the ratio of the relative local length $|\partial_u\mb{X}(u,t)|$ to the total length $L(\Gamma_t) = \int_0^1 |\partial_u\mb{X}(u,t)| du$:
\begin{equation}
\frac{\partial }{\partial t}
\frac{|\partial_u\mb{X}(u,t)|}{L(\Gamma_t)} = 
\frac{|\partial_u\mb{X}(u,t)|}{L(\Gamma_t)}
\left(
\partial_{s} \alpha  -\kappa \beta +  \langle  \kappa \beta\rangle \right), \quad\text{where}\ \langle  \kappa \beta\rangle = \frac{1}{L(\Gamma_t)} \int_{\Gamma} \kappa \beta ds.
\label{alpha}
\end{equation}
Therefore, the ratio $|\partial_u\mb{X}(u,t)|/L(\Gamma_t)$ is constant with respect to the time $t$, that is, 
\begin{equation}
\frac{|\partial_u\mb{X}(u,t)|}{L(\Gamma_t)} = \frac{|\partial_u\mb{X}(u,0)|}{L(\Gamma_0)},
\quad \text{for any} \ t\ge0, 
\label{uniformalpha}
\end{equation}
provided that the tangential velocity part $\alpha=\mb{F}\cdot\mb{T}$ satisfies $\partial_{s} \alpha = \kappa \beta -  \langle  \kappa \beta\rangle$ (see Fig.~\ref{fig:redis}). Another suitable choice of the tangential velocity $\alpha$ is the so-called asymptotically uniform tangential velocity proposed and analyzed by Mikula and \v{S}ev\v{c}ovi\v{c} in \cite{MS2004, MMAS2004}. If $\omega>0$ and
\begin{equation}
\partial_{s} \alpha = \kappa \beta  - 
 \langle  \kappa \beta\rangle  + \left( \frac{L(\Gamma_t)}{|\partial_u\mb{X}(u,t)|} - 1\right) \omega ,
\label{alpha-asymptotic}
\end{equation}
then, using (\ref{alpha}) we obtain
$\lim_{t\to \infty} \frac{|\partial_u\mb{X}(u,t)|}{L(\Gamma_t)} =1 $ uniformly with respect to $u\in[0,1]$ provided $\omega>0$ is a positive constant. This means that the redistribution becomes asymptotically uniform.

\begin{figure}
\begin{center}
\includegraphics[width=0.45\textwidth]{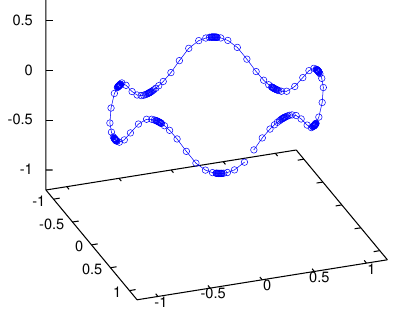}
\includegraphics[width=0.45\textwidth]{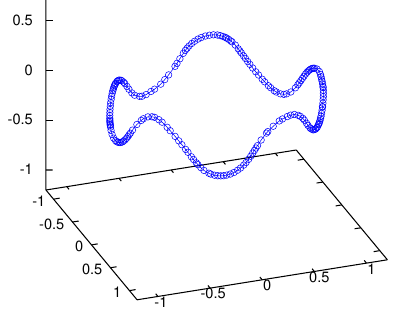}
\end{center}
\caption{Illustration of importance of a suitable choice of the tangential redistribution. Left: no tangential redistribution. Right: tangential redistribution preserving the relative local length (i.e., $\omega = 0$). }
\label{fig:redis}
\end{figure}

\section{Existence and uniqueness of classical H\"older smooth solutions}

In this section, we present results on the existence and uniqueness of the classical H\"older smooth solution to the system of equations (\ref{eq:ab}) that governs the evolution of a curve $\Gamma_t$ parameterized by $\mb{X}$ and a scalar quantitity $\varrho$. We follow the analytical approach developed by Bene\v{s}, Kol{\' a}{\v r}, \v{S}ev\v{c}ovi\v{c} \cite{BKS2022}. The main idea of the proof is based on the abstract theory of analytic semiflows in Banach spaces due to DaPrato and Grisvard \cite{daprato}, Angenent \cite{Angenent1990, Angenent1990b}, Lunardi \cite{Lunardi1984}. The proof of the local existence and uniqueness of a classical H\"older smooth solution is based on the analysis of the position vector equation (\ref{eq:ab}) in which we choose the uniform tangential velocity $\alpha$. It leads to a uniformly parabolic equation (\ref{eq:ab}) provided that the diffusion coefficient $a$ is uniformly bounded from below by a positive constant. In our proof, assumptions of strict positivity of the curvature $\kappa$ and the existence of the Frenet frame are not required. The main idea is to rewrite the system (\ref{eq:ab}) in the form of an initial value problem for the abstract parabolic equation:
\begin{equation}
\partial_t \Phi = \mathscr{F}(\Phi), \quad \Phi(0) = \Phi_0,
\label{abstratF}
\end{equation}
in a scale of Banach spaces. Furthermore, we have to show that for any $\tilde{\Phi}$, the linearization $\mathscr{A}=\mathscr{F}'(\tilde{\Phi} )$ generates an analytic semigroup and belongs to the so-called maximal regularity class. Since the source term $G$ may depend on the curvature $\kappa =|\partial_{s}\mb{X}\times\partial^2_{s} \mb{X} | = | \partial_{s}^2 \mb{X} |$. Thus, the parabolic equation for $\varrho$ may depend on the derivative of highest order $| \partial_{s}^2 \mb{X}|$. As a consequence, linearization $\mathscr{F}'(\tilde{\Phi})$ is a skew-symmetric linear operator having an off-diagonal term. 

First, we define the underlying function spaces. Assume that $0<\varepsilon<1$ and $k$ is a nonnegative integer. Let us denote by $h^{k+\varepsilon}(S^1)$ the so-called little H\"older space, that is, the Banach space, which is the closure of $C^\infty$ smooth functions in the norm of the Banach space of $C^k$ smooth functions defined in the periodic domain $S^1$, and such that the $k$-th derivative is $\varepsilon$-H\"older smooth. The norm is given as a sum of the $C^k$ norm and the H\"older semi-norm of the $k$-th derivative, that is,
\[
h^{k+\varepsilon}(S^1) =\{   \phi:S^1\to \mathbb{R}, \ \Vert \phi\Vert_{h^{k+\varepsilon}} = \sum_{j=0}^k \sup_{u\in S^1} |\partial^j_u \phi(u)| 
+  \sup_{u,v\in S^1}\frac{|\partial^k_u\phi(u) - \partial^k_v\phi(v)|}{|u-v|^\varepsilon}
<\infty \}
\]
Next, by $C^k([0,T], \mc{E})$ we denote the Banach space consisting of all $C^k$ continuous functions from the interval $[0,T]$ to the Banach space $\mc{E}$ with the norm
\[
\Vert\Phi\Vert_{C^k([0,T], \mc{E})} = \sum_{j=0}^k \max_{t\in[0,T]} \Vert \partial_t^j\Phi(t)\Vert_{\mc{E}}.
\]
For a given $0<\varepsilon <1$  we define the following scale of Banach spaces of $\varepsilon$-H\"older continuous functions defined in the periodic domain $S^1$:
\begin{equation}
E_k = h^{2k +\varepsilon}(S^1)\times h^{2k +\varepsilon}(S^1) \times h^{2k +\varepsilon}(S^1), \quad
\mc{E}_k = E_k \times h^{2k +\varepsilon}(S^1), \quad k=0, 1/2, 1 .
\label{Espaces}
\end{equation}
Recall the following continuous and compact embedding: $\mc{E}_{1}\hookrightarrow \mc{E}_{1/2}\hookrightarrow \mc{E}_{0}$. Furthermore, using the Gagliardo-Nirenberg interpolation inequality on the scale of H\"older spaces (see \cite[Theorem 1]{BM2018}), we have 
\begin{equation}
\Vert\Phi \Vert_{\mc{E}_{\frac12}} \le C \Vert\Phi \Vert_{\mc{E}_{1}}^{\frac12}
\Vert\Phi \Vert_{\mc{E}_{0}}^{\frac12}
\le \delta \Vert\Phi \Vert_{\mc{E}_{1}} + \frac{C^2}{4\delta} \Vert\Phi \Vert_{\mc{E}_{0}}, 
\quad
\text{for any} \ \delta>0 \ \text{and} \ \Phi\in \mc{E}_{1},
\label{interpolation}
\end{equation}
where $C>0$ is a constant depending on $\varepsilon\in(0,1)$ only. For $a,b\in h^{1+\varepsilon}(S^1)$ and $\mb{t}=(t_1,t_2,t_3)^T\in E_{\frac12}$ we define the matrix function:
\[
A(a,b,\mb{t}) = a I + b [\mb{t}]_\times :=
\left(
\begin{array}{ccc}
     a & -b t_3 & b t_2
     \\
     b t_3 & a & -b t_1
     \\
     -b t_2 & b t_1 & a
\end{array}
\right).
\]
Clearly,
$a \partial^2_{s} \mb{X} + b (\partial_{s}\mb{X}\times\partial^2_{s} \mb{X}) = A(a,b,\partial_s\mb{X}) \partial^2_{s} \mb{X}$. Then the system of equations (\ref{eq:ab})  
can be rewritten as follows:
\begin{equation}
\begin{split}
\partial_t \mb{X} &= A(a,b,\partial_s\mb{X})  \partial^2_s \mb{X} + \mb{F} (\mb{X}, \partial_{s}\mb{X} ,\varrho, \partial_s\varrho), \\
\partial_t \varrho &=  c \partial^2_s \varrho + G(\mb{X}, \partial_{s}\mb{X} ,\varrho, \partial_s\varrho, \kappa\beta).  
\end{split}
\label{system} 
\end{equation} 
Let us denote $\Phi= (\mb{X}, \varrho)^T$. In what follows, we shall assume that the mappings:
\begin{equation}
\begin{split}
&\mc{E}_{1/2} \ni \Phi \mapsto \mb{F} (\mb{X}, \partial_{s}\mb{X} ,\varrho, \partial_s\varrho) \in E_0,
\\
& \mc{E}_{1/2} \times h^\varepsilon(S^1) \ni (\Phi, \eta) \mapsto G (\mb{X}, \partial_{s}\mb{X} ,\varrho, \partial_s\varrho, \eta) \in h^\varepsilon(S^1),
\\
& \mc{E}_{1/2} \times h^{1+\varepsilon}(S^1) \ni (\Phi, \varrho) \mapsto 
d(\mb{X}, \partial_{s}\mb{X} ,\varrho, \partial_s\varrho) \in h^\varepsilon(S^1),\quad d\in\{a,b,c\},
\\
& \text{are $C^1$ continuous and globally Lipschitz continuous.}
\end{split}
\label{FG}
\end{equation}
For example, the function $G$ can be defined as in (\ref{eq:balance_differential}) with the uniform ($\omega=0$) or asymptotically uniform ($\omega>0$) tangential velocity $\alpha$ given by (\ref{alpha-asymptotic}) with $\omega=0$,  i.e., 
\[
G (\mb{X}, \partial_{s}\mb{X} ,\varrho, \partial_s\varrho, \eta) = -\varrho \left(  \partial_s \alpha - \eta \right) +  c \partial^2_s \varrho - \partial_s (\varrho v)  + q 
\]
where the tangential velocity $\alpha$ is calculated from the following equation:
\begin{equation}
\partial_{s} \alpha =  \eta -  \langle \eta\rangle  + \left( \frac{L(\Gamma_t)}{|\partial_u\mb{X}(u,\cdot)|} - 1\right) \omega,\ \text{where}\ \eta=\kappa \beta. 
\end{equation} 
Then the system of governing equations (\ref{system}) can be rewritten in the abstract form (\ref{abstratF}) where $\mathscr{F}$ is defined by the right-hand side of (\ref{system}). 
Suppose that $\tilde{\Phi}= (\tilde{\mb{X}}, \tilde{\varrho})^T\in \mc{E}_1$. 
Recall that $\kappa = |\partial^2_{s} \mb{X}|$ and $\beta = a \kappa + \mb{F} \cdot \mb{N}$, therefore $\eta:= \kappa\beta = a \kappa^2 + \mb{F} \cdot \kappa \mb{N} = a |\partial^2_{s} \mb{X}|^2 + \mb{F} \cdot \partial^2_{s} \mb{X} $ because $\kappa \mb{N} = \partial_s\mb{T} = \partial^2_{s} \mb{X}$. Therefore, the Fr\'echet derivative at $\tilde{\mb{X}}\in E_1$ of $\eta=\eta(\mb{X})$ in the direction $\mb{X}\in E_1$ has the form
\[
\eta^\prime(\tilde{\mb{X}}) \mb{X} = (2 \tilde{a} \tilde{\mb{X}} +  \tilde{\mb{F}} ) \cdot \partial^2_{s} \mb{X} 
+ \mathscr{L}[\mb{X}],
\]
where $\tilde{\mb{F}} = \mb{F}(\tilde{\mb{X}}, \partial_{s}\tilde{\mb{X}}, \tilde{\varrho}, \partial_s\tilde{\varrho}) \in E_0$, and $ \mathscr{L}:E_{1/2} \to E_0$ is a bounded linear operator containing lower-order derivatives of $\mb{X}$, including non-local dependencies. Hence, the principal part $\mathscr{A}_0$ of the linearization $\mathscr{F}'(\tilde{\Phi})$ containing the second-order derivatives has the matrix form $\partial_t\Phi = \mathscr{A}_0 \Phi$ where $\mathscr{A}_0$ is a skew-diagonal linear operator:
\[
\mathscr{A}_0 \Phi =
\left(
\begin{array}{cc}
     \tilde{A} & 0
     \\
     \tilde{d}^T  & \tilde{c}
\end{array}
\right) \partial_s^2 \Phi,
\]
where $\tilde a=a(\tilde{\mb{X}}, \partial_{s}\tilde{\mb{X}}, \tilde{\varrho}, \partial_s\tilde{\varrho}) > 0,\  
\tilde b=b(\tilde{\mb{X}}, \partial_{s}\tilde{\mb{X}}, \tilde{\varrho}, \partial_s\tilde{\varrho}), \ 
\tilde c=c(\tilde{\mb{X}}, \partial_{s}\tilde{\mb{X}}, \tilde{\varrho}, \partial_s\tilde{\varrho}) >0, \tilde{A} = A(\tilde{a},\tilde{b},\partial_s\tilde{\mb{X}})$, and $\tilde{d} = G^\prime_w(\tilde{\mb{X}}, \partial_{s}\tilde{\mb{X}}, \tilde{\varrho}, \partial_s\tilde{\varrho}, \tilde{\eta}) \eta^\prime(\tilde{\mb{X}})  \in h^\varepsilon(S^1)$. Clearly, $\tilde{a},\tilde{b},\tilde{c}, \tilde{A}_{ij}\in h^{1+\varepsilon}(S^1), \tilde{d}\in E_{1/2}$. 
The operator $\mathscr{A}_0$ generates an analytic semigroup $\{\exp(t\,\mathscr{A}_0), t\ge0\}$ in $_0$, 
\[
\exp(t\,\mathscr{A}_0)
=
\left(
\begin{array}{cc}
     \exp(t\,\tilde{A}\partial^2_s), & 0
     \\
    D(t), & \exp(t\,\tilde{c}\partial^2_s)
\end{array}
\right), 
\]
where Duhamel's integral $D$ is given by $$D(t)= \int_0^t \exp((t-\tau)\tilde{c}\partial^2_s) (\tilde{d}^T \partial^2_s) \exp(\tau\tilde{A}\partial^2_s) d\tau.$$ Next, we will prove that $\mathscr{A}_0$ belongs to the maximal regularity class ${\mathcal M}(\mc{E}_1,\mc{E}_0)$. This means that for any time interval $[0,T], T>0$, the mapping $\partial_t - \mathscr{A}_0: \Phi  \mapsto (\Phi_0, f)$ from $C([0,T], \mc{E}_1)\cap C^1([0,T], \mc{E}_0)$ to $ \mc{E}_1 \times C([0,T], \mc{E}_0)$, is invertible, where $f(t) = \partial_t \Phi(t) - \mathscr{A}_0 \Phi(t), \ \Phi_0=\Phi(0)$, i.e., the inverse operator $(\partial_t - \mathscr{A}_0 )^{-1}:  \mc{E}_1 \times C([0,T], \mc{E}_0)\to C([0,T], \mc{E}_1)\cap C^1([0,T], \mc{E}_0)$ exists and is a bounded linear operator. To prove this statement, assume $f=(f^{\mb{X}}, f^\varrho)\in  C([0,T], \mc{E}_0)=C([0,T], E_0)\times C([0,T], h^\varepsilon(S^1))$, and $\Phi_0=(\mb{X}_0, \varrho_0)\in \mc{E}_1=E_1\times h^{2+\varepsilon}(S^1)$. Assume $\tilde{a},\tilde{b}\in h^{1+\varepsilon}(S^1)$ and that the function $\tilde{a}$ is strictly positive, $T>0$. According to \cite[Proposition 3]{BKS2022}\label{maximalregularity} the operator $(\tilde{a} \pm i \tilde{b}) \partial^2_s $ belongs to the maximal regularity class ${\mathcal M}(h^{2+\varepsilon}(S^1), h^{\varepsilon}(S^1))$ on the time interval $[0,T]$. Moreover, if $\tilde{\mb{t}}=\partial_s\tilde{\mb{X}} \in E_{\frac12}$, then the linear operator $\tilde{A}\partial^2_s  = A(\tilde{a},\tilde{b},\partial_s\tilde{\mb{X}})\partial^2_s $  belongs to the maximal regularity class ${\mathcal M}(E_1, E_0)$ in the time interval $[0,T]$. This means that the operator $\partial_t - \tilde{A}\partial^2_s: \mb{X} \ni C([0,T], E_1)\cap C^1([0,T], E_0) \mapsto (\mb{X}_0, f^{\mb{X}}) \in E_1 \times C([0,T], E_0)$ is invertible. Similarly, the operator $\tilde{c} \partial^2_s $ belongs to the maximal regularity class ${\mathcal M}(h^{2+\varepsilon}(S^1), h^{\varepsilon}(S^1))$. Therefore, the inverse operator $(\partial_t - \mathscr{A}_0)^{-1}$ is given by $\Phi=(\mb{X},\varrho)=(\partial_t - \mathscr{A}_0)^{-1} (\Phi_0,f)$ is bounded. It can be expressed as $\mb{X}=(\partial_t - \tilde{A}\partial^2_s)^{-1} (\mb{X}_0,f^{\mb{X}})$ and $\varrho=(\partial_t - \tilde{c}\partial^2_s)^{-1} (\varrho_0, \xi)$ where $\xi=f^\varrho+\tilde{d}^T\partial^2_s \mb{X} \in C([0,T], h^\varepsilon(S^1))$.  As a consequence, $\mathscr{A}_0$ belongs to the maximal regularity class ${\mathcal M}(\mc{E}_1,\mc{E}_0)$. 

We decompose the linearization $\mathscr{F}'(\tilde{\Phi}) = \mathscr{A}_0 + \mathscr{A}_1$ where the operator $\mathscr{A}_1= \mathscr{F}'(\tilde{\Phi}) - \mathscr{A}_0$ depends on lower-order derivatives $\partial^k_s, k=0,1$, of $\tilde{\Phi}$ in the sense that it is a bounded linear operator $\mathscr{A}_1: \mc{E}_{1/2}\to \mc{E}_0$. Therefore, the operator $\mathscr{A}_1$ considered as a mapping from $\mc{E}_1\to \mc{E}_0$ has the relative zero norm with respect to $\mathscr{A}_0$ . Therefore, the linearization $\mathscr{F}'(\tilde{\Phi})$ belongs to the maximal regularity class ${\mathcal M}(\mc{E}_1,\mc{E}_0)$ because this class is closed with respect to perturbation with the relative zero norm (cf.  \cite[Lemma 2.5]{Angenent1990}, DaPrato and Grisvard \cite{daprato}, Lunardi \cite{Lunardi1984}).

The method based on analytic semigroup and maximal regularity theory has been successfully applied to prove the existence, regularity, and uniqueness of solutions representing evolving families of 2D and 3D curves in the series of by \v{S}ev\v{c}ovi\v{c}, Mikula, Yazaki, Bene\v{s} and Kol\'a\v{r}
\cite{sevcovic2001evolution}, \cite{MS2004}, \cite{MMAS2004}, \cite{SevcovicYazaki2012}, 
\cite{BKS2017}, \cite{BKS2020}, \cite{BKS2022}.

By a solution to the system of nonlinear equations (\ref{eq:ab}) that satisfies the initial condition $\Phi_0=(\mb{X}_0, \varrho_0)\in E_1\times h^{2+\varepsilon}(S^1)$ we mean a function $\Phi=(\mb{X},\varrho)\in  C([0,T], \mc{E}_1) \cap C^1([0,T], \mc{E}_0)$ such that $\Phi(\cdot, 0)= (\mb{X}_0, \varrho_0)$. Now, we can state the following result on the local existence and uniqueness of solutions.

\begin{theorem}\label{theo-main}
Assume that the functions $\mb{F}, G, a>0, b, c>0,$ satisfy hypothesis (\ref{FG}) and that the tangential velocity $\alpha$ preserves the relative local length and is given by (\ref{uniformalpha}). Assume that the parameterization $\mb{X}_0=\mb{X}(\cdot,0)$ of the initial curve $\Gamma_0$ belongs to the space $E_1$, and the initial scalar quantity $\varrho_0=\varrho(\cdot, 0)$ belongs to the space $h^{2+\varepsilon}(S^1)$. Then there exists $T>0$ and a unique solution $\Phi=(\mb{X},\varrho)\in  C([0,T], \mc{E}_1) \cap C^1([0,T], \mc{E}_0)$ such that $\Phi(\cdot, 0)= (\mb{X}_0, \varrho_0)$ to the system of parabolic equations (\ref{eq:ab}). 

\end{theorem}

\begin{proof}
The proof follows from the abstract result on the existence and uniqueness of solutions to (\ref{abstratF}) due to Angenent \cite{Angenent1990}. It is based on the linearization of the abstract evolution equation (\ref{abstratF}) in the Banach space $\mc{E}_1$. We showed that, for any $\tilde{\Phi}$, the linearization $\mathscr{F}'(\tilde{\Phi} )$ generates an analytic semigroup and belongs to the maximal regularity class ${\mathcal M}(\mc{E}_1,\mc{E}_0)$ of linear operators from the Banach space $\mathcal{E}_1$ to the Banach space $\mathcal{E}_0$. The local existence and uniqueness of a solution $\Phi=(\mb{X},\varrho)\in  C([0,T], \mc{E}_1) \cap C^1([0,T], \mc{E}_0)$, and $\Phi(\cdot, 0)= \Phi_0$ now follow from the abstract result \cite[Theorem 2.7]{Angenent1990} due to Angenent.
\end{proof}

\begin{corollary}
Suppose that all assumptions of Theorem~\ref{theo-main} are satisfied. Then, for any initial closed curve $\Gamma_0=\{\mb{X}_0(u), u\in[0,1]\}$  and the initial scalar quantity $\hat\varrho_0(\mb{X})=\varrho_0(u), \mb{X}=\mb{X}_0(u)$ where $\mb{X}_0\in E_1$ and $\varrho_0\in h^{2+\varepsilon}(S^1)$, there exist $T>0$, and a unique family of closed curves $\Gamma_t=\{\mb{X}(u,t), u\in[0,1]\}, t\in[0,T]$, evolving according to the governing system of equations (\ref{eq:ab}). The scalar quantity (density) $\hat\varrho(t, \mb{X})$ evaluated at the point $\mb{X}\in\Gamma_t$ is unique with respect to $t\in[0,T]$ and $\mb{X}\in\Gamma_t$.
\end{corollary}

\begin{proof}
With regard to Theorem~\ref{theo-main} the solution $\Phi=(\mb{X}, \varrho)$ to the system of parabolic equations (\ref{eq:ab}) satisfying a given initial condition $\Phi_0=(\mb{X}_0, \varrho_0)$ is unique in the Banach space $C([0,T], \mc{E}_1) \cap C^1([0,T], \mc{E}_0)$. Since the evolving family of closed curves $\Gamma_t=\{\mb{X}(u,t), u\in[0,1]\}\subset \mathbb{R}^3$ is independent of reparameterization induced by a particular choice of tangential velocity $\alpha$ then the family of closed curves $\Gamma_t$, $t\in[0,T]$, is also unique.

Although the scalar quantity $\varrho(t,u)$ depends on a particular parameterization of $\Gamma_t$, the physical scalar quantity (density) $\hat\varrho(t, \mb{X})=\varrho(t,u)$ evaluated at the point $\mb{X}=\mb{X}(t,u)$ and satisfying the initial condition $\hat\varrho(0, \mb{X}) = \hat\varrho_0(\mb{X})=\varrho_0(u)$, where $\mb{X}=\mb{X}_0(u)$, is unique with respect to $t\in[0,T]$ and $\mb{X}\in\Gamma_t$.
\end{proof}

\section{Semi-analytical examples of periodic solutions}

In this section, we present two examples of the evolution of a circular curve with no torsion evolving in rotating and parallel planes. The motion of the curves is coupled by a scalar quantity $\varrho$. The method of construction of a solution in a special separated form and its reduction to solving a system of ODEs has been proposed in the recent paper by Bene\v{s}, Kol\'a\v{r} and \v{S}ev\v{c}ovi\v{c} in the context of evolving 3D curves with interactions \cite{BKS2022}. 
We consider the following system of coupled equations governing the motion of curves in 3D and scalar quantity:
\begin{equation}
\label{example-system}
\begin{split}
 &\partial_t \mb{X} = \beta \mb{N}  + \gamma \mb{B}  + \alpha \mb{T}, \\
 & \partial_t \varrho + (\partial_s\alpha - \kappa\beta)\varrho = \partial_s^2\varrho + q.
\end{split}
\end{equation}
For simplicity, we assume no tangential velocity in (\ref{example-system}), i.e. $\alpha=0$. The normal velocity $\beta$, the binormal velocity $\gamma$, and source term $q$ are assumed to depend on the shape of the curve $\Gamma_t=\{\mb{X}(u,t), u\in I\}$ and the scalar quantity $\varrho$. In this example, we shall assume a specific choice:
\begin{equation}
\beta = \kappa - \frac{2\pi}{L(\Gamma_t)} - \tilde{P}(L(\Gamma_t), \Vert\varrho\Vert_2), \quad  q = \frac{\varrho}{\sqrt{2}\Vert\varrho\Vert_2}, 
\label{example-beta}
\end{equation}
where $\tilde{P}$ is a smooth function and $\Vert\varrho\Vert_2 =\left(\int_0^1\varrho(u)^2 du\right)^\frac12$ is the $L^2$-norm of $\varrho$.

If the family  $\{\Gamma_t, t\ge 0\}$  of curves evolves in parallel planes, then the area $A(\Gamma_t)$ enclosed by the curve $\Gamma_t$ satisfies $\frac{d}{dt} A(\Gamma_t) = -\int_{\Gamma} \beta  ds $ (see \cite[Prop. 2]{BKS2022}). Since $\int_{\Gamma}\kappa ds = 2\pi$ for a curve evolving in the plane, we have $\int_{\Gamma} \beta  ds = 0$ provided that $\beta = \kappa - \frac{2\pi}{L(\Gamma_t)}$ (see Gage \cite{Gage86}). In such a case, the enclosed area $A(\Gamma_t)$ is preserved during evolution. In the next two examples, we will show that a perturbation of the normal velocity $\beta = \kappa - \frac{2\pi}{L(\Gamma_t)}$ and consideration of the influence of the scalar quantity $\varrho$ may lead to periodic motion of the curves and the scalar quantity. 

\subsection{An example of coupled system of parabolic equations for evolution of a curve and scalar quantity illustrating a supercritical Hopf bifurcation}

In the first example, we construct the solution $\mb{X}$ in the form of circular curves with radius $r$ evolving in a rotating plane. Define the vectors $\hat{\mb{X}}(u) = (\cos 2\pi u, \sin 2\pi u, 0)^T$ and $\hat{\mb{X}}^\perp(u) = (-\sin 2\pi u, \cos 2\pi u, 0)^T$.  Assume that a $3\times 3$ orthogonal matrix $Q$ and a positive scalar $r>0$ are time-dependent functions. Let the family of evolving curves $\Gamma_t=\{\mb{X}(u,t), u\in I, t\ge 0\}$ be defined as follows:
\begin{equation}
\mb{X}(u,t) = r(t) Q(t) \hat{\mb{X}}(u) .
\label{example-X}
\end{equation}
Then, for the curvature $\kappa$, arc-lenth parametrization $s$, tangent $\mb{T}$, normal $\mb{N}$ and binormal $\mb{B}$ vectors we have the following:
\[
\kappa=\vert\partial_s^2\mb{X}\vert = 1/r, \quad ds=2\pi r\, du, \quad \mb{T}= Q \hat{\mb{X}}^\perp, \quad \mb{N}= -Q \hat{\mb{X}}, \quad \mb{B}= Q e_3, \quad e_3=(0,0,1)^T.
\]
Assume that the function $\omega(t)$ satisfies $r(t)\frac{d\omega}{dt}(t)=1$. Then, it is easy to verify that if
\[
Q(t) = \begin{pmatrix}
\cos \omega(t) & 0  & -\sin \omega(t)\\
0 & 1 & 0 \\
\sin \omega(t) & 0  &  \cos \omega(t)\\
\end{pmatrix}
\ \ \text{then}\ \  Q Q^T = Q^T Q =I, \ \ 
r \frac{d}{dt}Q = Q \begin{pmatrix}
0 & 0 & -1\\
0 & 0 & 0 \\
1 & 0 & 0\\
\end{pmatrix} .
\] 
Therefore, $$\partial_t\mb{X} = \frac{d}{dt}(r Q) \hat{\mb{X}} = \frac{dr}{dt} Q \hat{\mb{X}} + r \frac{d}{dt}Q \hat{\mb{X}} = - \frac{dr}{dt} \mb{N} + \gamma Q e_3 = \beta \mb{N} + \gamma \mb{B},$$ where $\gamma=\cos 2\pi u $ and $\beta = - \frac{dr}{dt}$. We search for the solution $\varrho$ in separated form:
\begin{equation}
    \varrho(u,t)=a(t) \cos 2\pi u, 
\label{example-rho}
\end{equation}
where $a=a(t)\ge 0$ is a nonnegative amplitude. As $\Vert\varrho(\cdot,t)\Vert_2 =\left(\int_0^1 a(t)^2 (\cos 2\pi u)^2 du\right)^\frac12 = \frac{a(t)}{\sqrt{2}}$ we have $\gamma=q =  \frac{\varrho}{\sqrt{2}\Vert\varrho\Vert_2} = \cos 2\pi u$.  Therefore, for the normal velocity $\beta$ given by (\ref{example-beta}) we have $\beta = 1/r - 2\pi/(2\pi r) -  \tilde{P}(\Gamma_t, \varrho) =  - P(r, a)$, where $P(r,a) = \tilde{P}(L(\Gamma_t), \Vert \varrho\Vert_2)$. 
Since $ds = 2\pi r du$ we have $\partial_s^2 \varrho = (2\pi r)^2 \partial_u^2 \varrho = - a r^{-2}\cos 2\pi u$. Therefore, 
\[ \partial _t \varrho - \kappa\beta\varrho - \partial_s^2\varrho 
=
\left( \frac{d}{dt} a + \frac{a}{r} P(r,a)+ \frac{a}{r^2} \right) \cos 2\pi u.
\]
As a consequence, the pair $(\mb{X}, \varrho)$ is a solution to the coupled evolutionary equation (\ref{example-system}) if and only if the radius $r$ and the amplitude $a$ satisfy the planar ODE system:

\begin{equation}
\label{example-system2D}
 \frac{d}{dt} r = P(r, a), \qquad
 \frac{d}{dt} a = -\frac{a}{r} P(r,a) -  \frac{a}{r^2}  + 1 .
\end{equation}
The system (\ref{example-system2D}) of ODEs has a non-trivial steady state $(\tilde{r},\tilde{a})$, i.e. the time-independent solution,  where 
\[
\tilde{a} = \tilde{r}^2, \qquad P(\tilde{r}, \tilde{r}^2) =0. 
\]
Linearization $\mathcal{A}$ of the right-hand3 side of (\ref{example-system2D}) in steady state $(\tilde{r},\tilde{a})$ has the form of a $2\times 2$ matrix $\mathcal{M}$:
\[
\mathcal{A} = 
\begin{pmatrix}
\frac{\partial P}{\partial r} (\tilde{r}, \tilde{a}) &  \frac{\partial P}{\partial a} (\tilde{r}, \tilde{a}) &
\\
\\
- \frac{\tilde{a}}{\tilde{r}} \frac{\partial P}{\partial r} (\tilde{r}, \tilde{a}) 
+ 2 \frac{\tilde{a}}{\tilde{r}^3} &  - \frac{\tilde{a}}{\tilde{r}} \frac{\partial P}{\partial a} (\tilde{r}, \tilde{a}) -\frac{1}{\tilde{r}^2}
\end{pmatrix} .
\]
Now, as $\tilde{a} = \tilde{r}^2$ we have
\[
trace(\mathcal{A}) = \frac{\partial P}{\partial r} (\tilde{r}, \tilde{a}) 
- \tilde{r} \frac{\partial P}{\partial a} (\tilde{r}, \tilde{a})  -\frac{1}{\tilde{r}^2}, \qquad
det(\mathcal{A}) = -\frac{1}{\tilde{r}^2}\left( 
\frac{\partial P}{\partial r} (\tilde{r}, \tilde{a}) + 2 \tilde{r} \frac{\partial P}{\partial a} (\tilde{r}, \tilde{a}) 
\right).
\]
Assume that the source term $\tilde{P}$ has the specific form: 
\[
\tilde{P}(L(\Gamma_t), \Vert\varrho\Vert_2) = \frac{L^2}{4\pi^2} - \lambda\sqrt{2}\Vert\varrho\Vert_2 +1, \quad \text{that is},\quad P(a,r)= r^2 - \lambda a +1,
\]
where $\lambda>1$ is a parameter. Then, for the steady-state solution $(\tilde{r},\tilde{a})$ we have $\tilde{r}=1/\sqrt{\lambda-1}, \tilde{a}=1/(\lambda-1)$. 
\[
trace(\mathcal{A}) = \frac{\lambda+2}{\sqrt{\lambda-1}} - \lambda +1, \qquad
det(\mathcal{A}) = (\lambda-1)^{3/2}.
\]
The nonlinear equation $trace(\mathcal{A}) = 0$ has a unique solution for the parameter value $\lambda_0=4.473402$. Since $trace(\mathcal{A}) < 0$ for $\lambda>\lambda_0$, and $trace(\mathcal{A}) > 0$ for $\lambda<\lambda_0$. Hence, the supercritical Hopf bifurcation occurs when the parameter $\lambda$ crosses the critical value $\lambda_0$. From a stable focus steady state $(\tilde{r}, \tilde{a})$ for $\lambda>\lambda_0$ we observe a bifurcation to a stable period orbit at $\lambda_0$ that persists to exist for $\lambda<\lambda_0$ (see Fig.~\ref{fig-hopf} and Fig.~\ref{fig-hopf-3D-2}).

\begin{figure}
\begin{center}
\includegraphics[width=0.32\textwidth]{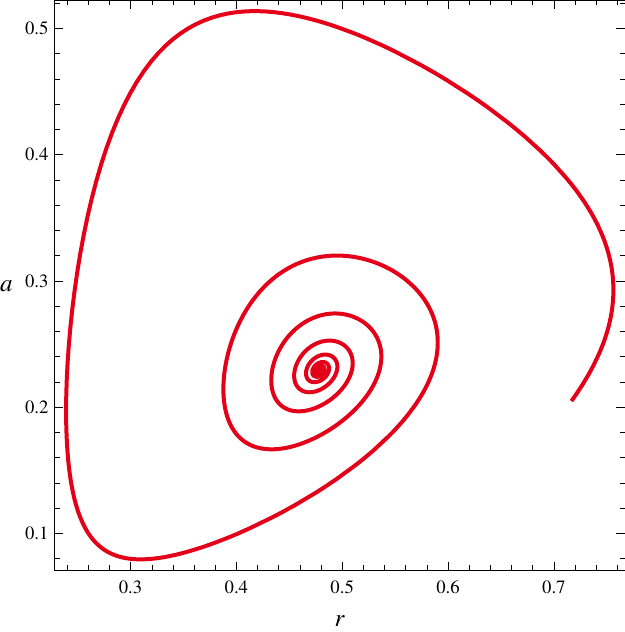}
\includegraphics[width=0.32\textwidth]{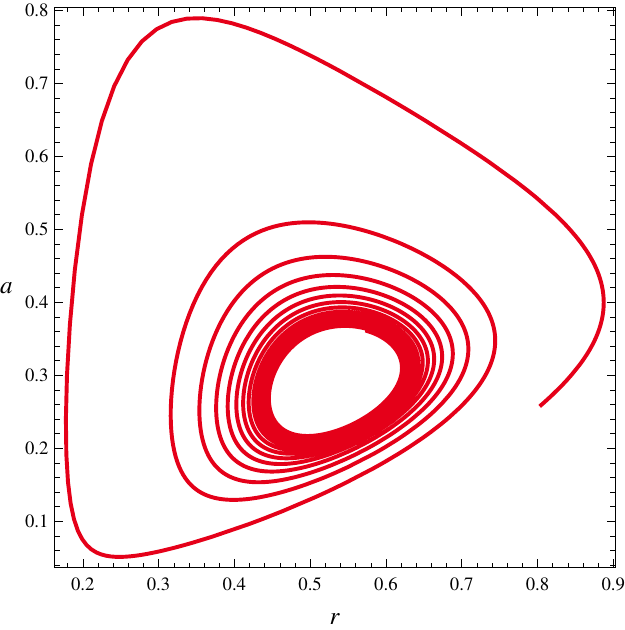}
\includegraphics[width=0.32\textwidth]{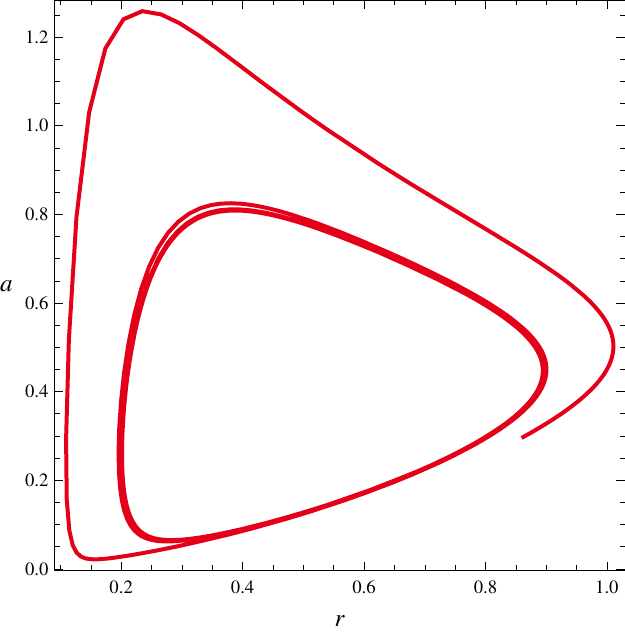}

\includegraphics[width=0.32\textwidth]{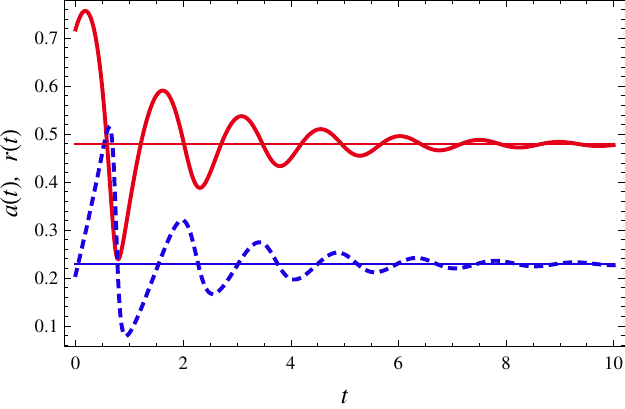}
\includegraphics[width=0.32\textwidth]{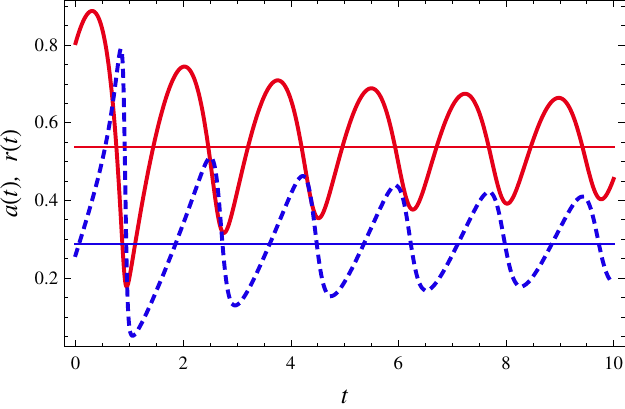}
\includegraphics[width=0.32\textwidth]{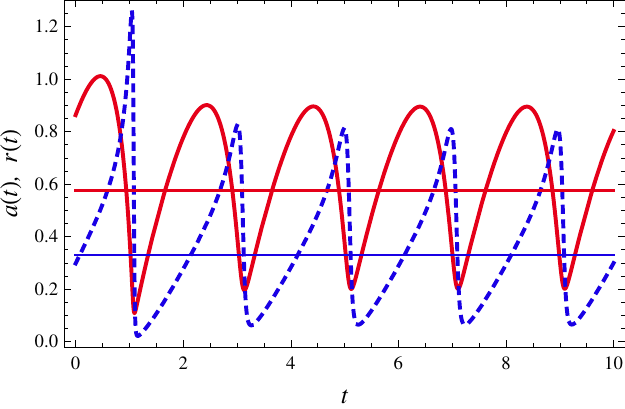}

\caption{\small Top row: phase portraits of solutions of (\ref{fig-hopf}). Bottom row: time dependent trajectories of solutions for the parameter values $\lambda=5.36808$ (left), $\lambda= \lambda_0= 4.473402$ (middle), and $\lambda=4.02606$ (right).
}

\label{fig-hopf}
\end{center}
\end{figure}

\subsection{Another  example of a coupled system of parabolic equations for evolution of a curve and scalar quantity}

In this example, we will construct a family of curves evolving in parallel planes and a scalar quantity $\varrho$
such that it solves the coupled system of evolution equations (\ref{example-system}) with binormal velocity $\gamma = -\beta$.
Assume that the family of circular curves $\Gamma_t=\{\mb{X}(u,t), u\in[0,1]\}$ evolves in parallel planes and is parameterized by
\[
\mb{X}(u,t) = (r(t) \cos 2\pi u, r(t) \sin 2\pi u, r(t))^T, \qquad \varrho(u,t)= a(t) \cos 2\pi u.
\]
Similarly to the previous example, we consider the normal velocity $\beta$ and the external source term $q$ of the form (\ref{example-beta}). Since $\mb{N}= - (\cos 2\pi u, \sin 2\pi u, 0)^T$ and $\mb{B}=(0,0,1)^T$ it is easy to verify that $(\mb{X}, \varrho)$ is a solution to (\ref{example-system}) with binormal velocity $\gamma=-\beta$ if and only if $\frac{d}{dt} r = - \beta$ and amplitude $a$ satisfies the ODE: $\frac{d}{dt} a - \frac{a}{r} \beta + \frac{a}{r^2}  =1$. Therefore, for the normal velocity $\beta$ given by (\ref{example-beta}) we have $\beta = 1/r - 2\pi/(2\pi r) -  \tilde{P}(\Gamma_t, \varrho) =  - P(r, a)$, where $P(r,a) = \tilde{P}(\Gamma_t, \Vert\varrho\Vert_2) =  r^2 - \lambda a +1$. It means that the pair of functions $(r,\varrho)$ satisfies the system of ODEs (\ref{example-system2D}) that exhibits a Hopf supercritival bifurcation at the parametrix value
$\lambda=4.02606<\lambda_0= 4.473402$. In Fig.~\ref{fig-hopf-3D-1} we present an example of the evolution of concentric circles with radius $r$ converging to $\tilde{r}=1/\sqrt{\lambda-1}$ and the amplitude of the scalar quantity $a$ converging to $\tilde{a}=1/(\lambda-1)$ for the parameter value $\lambda=5.36808>\lambda_0$.

\begin{figure}
\begin{center}
\includegraphics[width=0.24\textwidth]{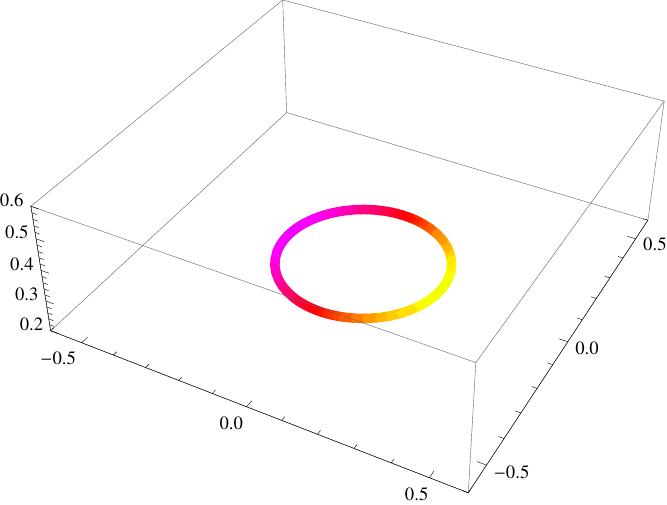}
\includegraphics[width=0.24\textwidth]{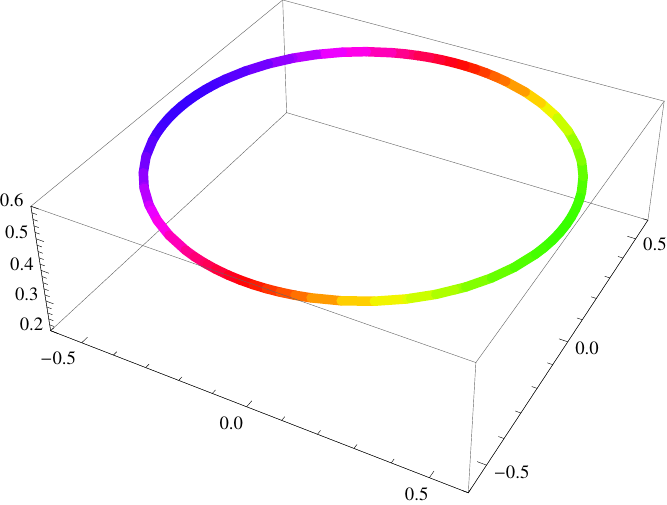}
\includegraphics[width=0.24\textwidth]{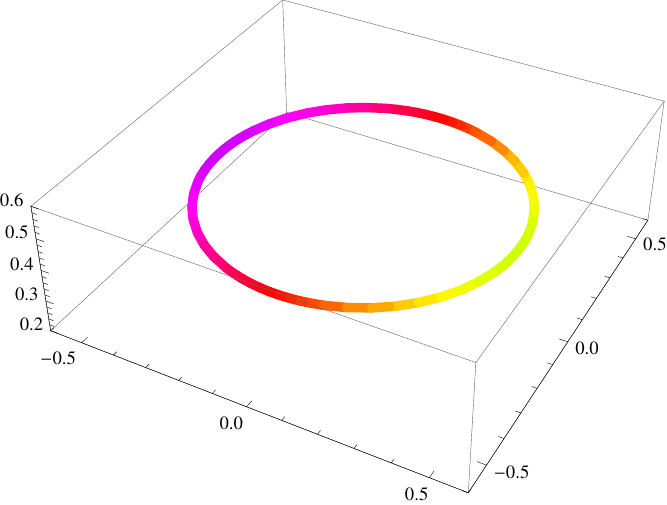}
\includegraphics[width=0.24\textwidth]{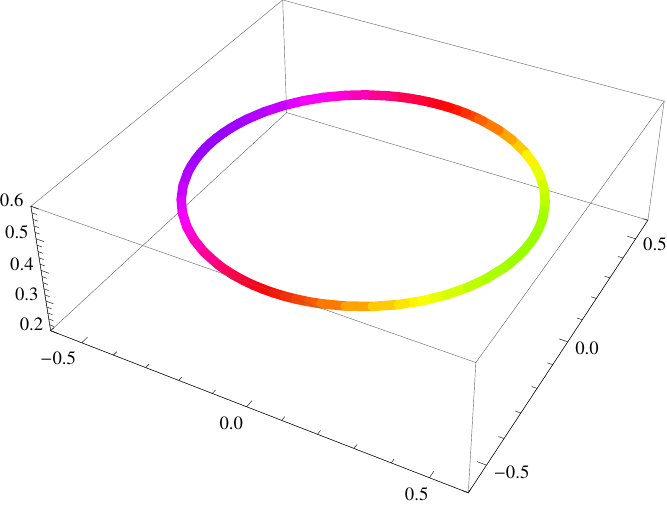}

{\scriptsize $t=0.8$ \hglue 3truecm $t=1.8$ \hglue 3truecm $t=4$ \hglue 3truecm $t=8$ }

\caption{\small Evolution of circular curves $\Gamma_t$ and the scalar quantity $\varrho(\cdot,t)$ (displayed as a colored field) that solve the coupled system (\ref{example-beta}) with the parameter value $\lambda=5.36808>\lambda_0$.}

\label{fig-hopf-3D-1}
\end{center}
\end{figure}

\begin{figure}
\begin{center}
\includegraphics[width=0.24\textwidth]{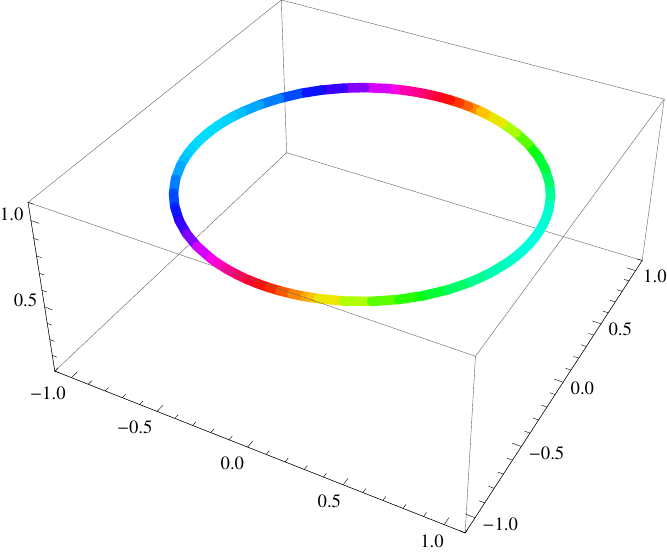}
\includegraphics[width=0.24\textwidth]{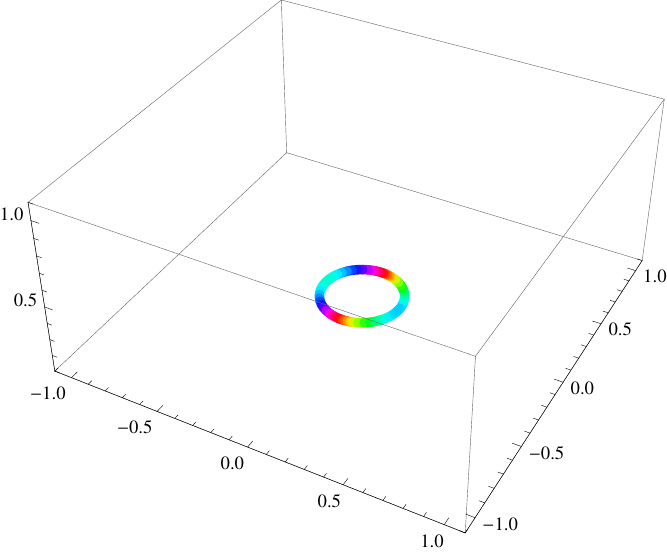}
\includegraphics[width=0.24\textwidth]{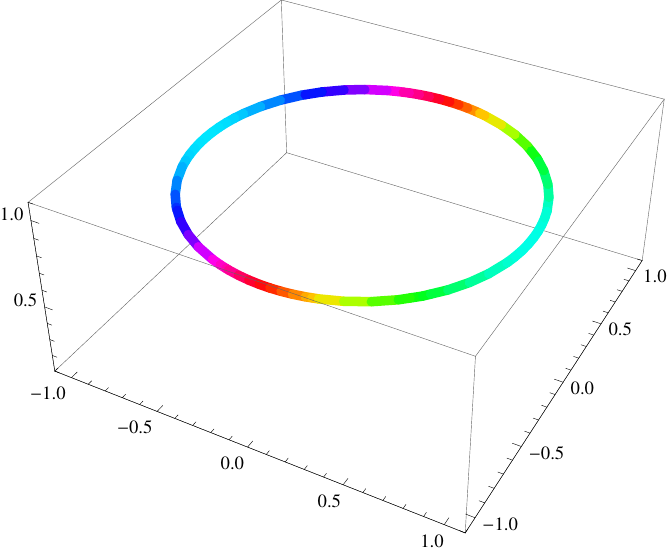}
\includegraphics[width=0.24\textwidth]{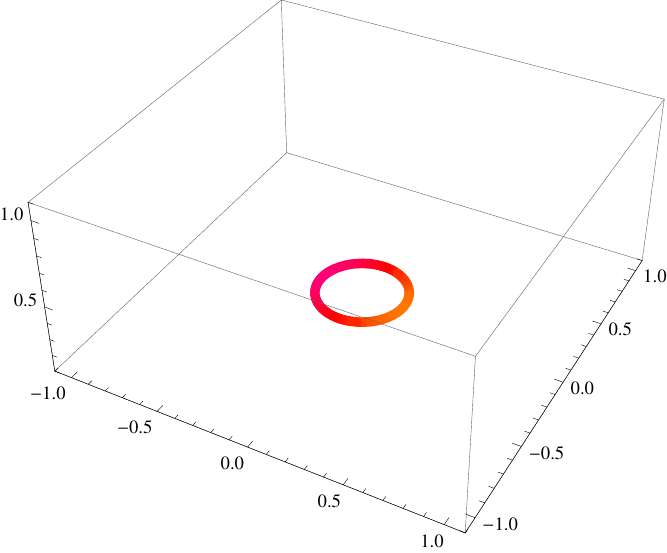}

{\scriptsize $t=2.5$ \hglue 3truecm $t=3.1$ \hglue 3truecm $t=4.4$ \hglue 3truecm $t=5.2$ }
\caption{\small Evolution of circular curves $\Gamma_t$ and the scalar quantity $\varrho(\cdot,t)$ (displayed as a coloured field)that solve  the coupled system (\ref{example-beta}) with the parameter value $\lambda=4.02606<\lambda_0$.}
\label{fig-hopf-3D-2}
\end{center}
\end{figure}

\section{Numerical discretization scheme based on the method of flowing finite volumes and the method of lines} \label{Sec:Num}
In this section, we present a numerical discretization scheme to solve the system of equations (\ref{eq:ab}). Our full space-time discretization scheme is based on a combination of the method of lines and the flowing finite-volume method of spatial discretization. The flowing finite-volume discretization was proposed by Mikula and \v{S}ev\v{c}ovi\v{c} \cite{sevcovic2001evolution} for the evolution of curves in the plane. It was further generalized and analyzed for evolving curves in 3D by Bene\v{s}, Kol\'a\v{r} and \v{S}ev\v{c}ovi\v{c} in \cite{BKS2020, BKS2022}. 

\subsection{Method of flowing finite volumes} \label{m-ffv} 
In this section, we describe the discretization method in more detail. We consider $M$ discrete nodes $\mb{X}_k$ approximating points $\mb{X}(u_k)$,  $k = 0,1, \ldots, M$, along the curve $\Gamma_t=\{\mb{X}(u), u\in I\}$, $u_k=k/M$. Similarly, we approximate the values of the scalar quantity $\varrho_k = \varrho(u_k)$. The corresponding dual nodes are defined as $\mb{X}_{k \pm \frac12} = \mb{X}(u_{k \pm \frac12})$  (see Fig.~\ref{FVMfig}). Here, $ u_{k \pm \frac12} = u_k \pm \frac{h}2$, where $h=1/M$, and $(\mb{X}_k + \mb{X}_{k+1}) / 2$ denote averages on segments connecting discrete nodes in the vicinity. It may differ from the position $\mb{X}_{k +\frac12} \in \Gamma_t$. The $k$-th segment between $\mb{X}_{k-1}$ and $\mb{X}_{k}$ represents the finite volume, while the segment between $\mb{X}_{k-1/2}$ and $\mb{X}_{k+1/2}$ represents the dual finite volume. Integration of equations (\ref{eq:ab}) for $(\mb{X},\varrho)$ in the dual volume between the dual nodes $\mb{X}_{k - \frac12}$ and $\mb{X}_{k + \frac12}$ yields the following result:
\begin{equation}
\label{eq:num_int}
\begin{split}
& \int_{u_{k-\frac12}}^{u_{k+\frac12}} \partial_t \mb{X} |\partial_u \mb{X}| d u = 
\int_{u_{k-\frac12}}^{u_{k+\frac12}} a\frac{\partial}{\partial u} \left( \frac{\partial_u \mb{X}}{|\partial_u \mb{X}|} \right) + b (\partial_s \mb{X} \times \partial_s^2 \mb{X}) |\partial_u \mb{X}| + \mb{F} |\partial_u \mb{X}| \ d u,
\\
& \int_{u_{k-\frac12}}^{u_{k+\frac12}} \partial_t \varrho |\partial_u \mb{X}| d u = 
\int_{u_{k-\frac12}}^{u_{k+\frac12}} c\frac{\partial}{\partial u} \left( \frac{\partial_u \varrho}{|\partial_u \mb{X}|} \right)  + G |\partial_u \mb{X}| \ d u .
\end{split}
\end{equation}
We denote $d_k = |\mb{X}_k - \mb{X}_{k-1}|$, $d_{M+1}=d_1$,  and $d_{k+\frac12} =(d_{k+1} + d_k)/2$ for $k=0, 1,\ldots,M$, where $\mb{X}_0 = \mb{X}_M$ and $\mb{X}_1 = \mb{X}_{M+1}$ for a closed curve $\Gamma_t$. We approximate the integral terms in (\ref{eq:num_int}) using the flowing finite-volume method. We assume that the quantities $\partial_t \mb{X}, \partial_u \mb{X}, \mb{F}, \partial_t \varrho, \partial_u \varrho$ and $G$ are constant in the finite volume between the nodes $\mb{X}_{k + \frac12}$ and $\mb{X}_{k - \frac12}$, and taking the values $\partial_t \mb{X}_k, \partial_u \mb{X}_k, \mb{F}_k, \partial_t \varrho_k, \partial_u \varrho_k$, and $G_k$ an respectively. Then the approximation of the terms of the first equation in (\ref{eq:num_int}) reads as follows:
\begin{equation}
\begin{split}
& \int_{u_{k-\frac12}}^{u_{k+\frac12}} \partial_t \mb{X} |\partial_u \mb{X}| d  u
\approx \frac{d  \mb{X}_k}{d t} d_{k+\frac12},
\\
& \int_{u_{k-\frac12}}^{u_{k+\frac12}}
a \partial_u \left( \frac{\partial_u \mb{X}}{|\partial_u \mb{X}|} \right) d u
\approx
a_k \left(\frac{\mb{X}_{k+1} - \mb{X}_k}{d_{k+1}} - \frac{\mb{X}_k - \mb{X}_{k-1}}{d_k}\right),
\\
& \int_{u_{k-\frac12}}^{u_{k+\frac12}}
b (\partial_s \mb{X} \times \partial_s^2 \mb{X}) |\partial_u \mb{X}| d u
\approx b_k d_{k+\frac12} \kappa_k \mb{B}_k,
\ \ 
\int_{u_{k-\frac12}}^{u_{k+\frac12}} \mb{F} |\partial_u \mb{X}| d u
\approx  \mb{F}_k  d_{k+\frac12}.
\end{split}
\label{discretization}
\end{equation}

\begin{figure}
\begin{center}
\includegraphics[width=0.3\textwidth]{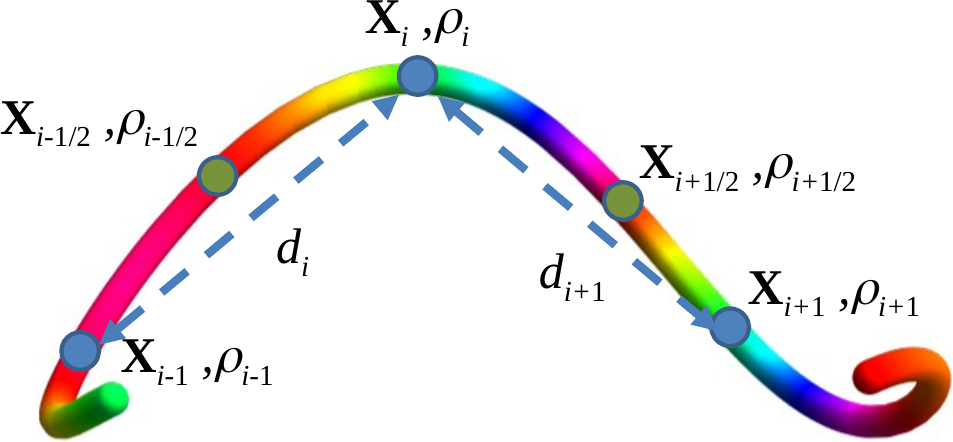}
\end{center}
\caption{Discretization of a segment of a 3D curve by the method of flowing finite volumes.
}
\label{FVMfig}
\end{figure}

Analogously, we can approximate the terms entered into the second equation in (\ref{eq:num_int}) for the scalar quantity $\varrho$. Approximations of the non-negative curvature $\kappa = |\partial^2_s \mb{X}|$, tangent vector $\mb{T}=\partial_s \mb{X}$, normal vector $\mb{N} = \partial^2_s \mb{X} / |\partial^2_s \mb{X} |$ and binormal vector $\mb{B}=\mb{T}\times\mb{N}$ read as follows:
\begin{equation}
\label{krivost}
\begin{split}
&\kappa_k  \approx  
\left|\frac{2}{d_k + d_{k+1}}
\left(\frac{\mb{X}_{k+1} - \mb{X}_k}{d_{k+1}} - \frac{\mb{X}_k - \mb{X}_{k-1}}{d_k}\right)
\right|, 
\  \mb{T}_k \approx \frac{\mb{X}_{k+1} - \mb{X}_{k-1}}{d_{k+1}+d_k}, 
\\
& \mb{N}_k  \approx 
\left(\frac{\mb{X}_{k+1} - \mb{X}_k}{d_{k+1}} - \frac{\mb{X}_k - \mb{X}_{k-1}}{d_k}\right)
\big/
\left|\frac{\mb{X}_{k+1} - \mb{X}_k}{d_{k+1}} - \frac{\mb{X}_k - \mb{X}_{k-1}}{d_k} \right|.
\end{split}
\end{equation}
In approximation $\mb{F}_k$ of the vector-valued function $\mb{F}$, we assume that the curve $\Gamma_t$ entering the definition of $\mb{F}$ is approximated by a polygonal curve with vertices $(\mb{X}_0, \mb{X}_1, \ldots, \mb{X}_M)$. The semi-discrete scheme for solving (\ref{eq:ab}) can be written as follows:
\begin{equation}
\label{eq:DCScheme}
\begin{split}
\frac{d \mb{X}_k}{d t}
&= a_k \frac{1}{d_{k+\frac12}}\left(\frac{\mb{X}_{k+1} - \mb{X}_k}{d_{k+1}} - \frac{\mb{X}_k - \mb{X}_{k-1}}{d_k}\right) + b_k  \kappa_k \mb{B}_k + \mb{F}_k,
\\
\frac{d \varrho_k}{d t}
&= a_k \frac{1}{d_{k+\frac12}}\left(\frac{\varrho_{k+1} - \varrho_k}{d_{k+1}} - \frac{\varrho_k - \varrho_{k-1}}{d_k}\right) + G_k,
\\
\mb{X}_k(0) &= \mb{X}_{ini}(u_k), \quad \varrho_k(0) = \varrho_{ini}(u_k), \quad \text{for} \ k = 1, \ldots, M.
\end{split}
\end{equation}
The resulting system (\ref{eq:DCScheme}) of ODEs  is numerically solved by means of the 4th-order explicit Runge-Kutta-Merson scheme with automatic time-stepping control and the tolerance parameter $10^{-3}$ (see \cite{0965-0393-24-3-035003}). We chose the initial time step as $4h^2$, where $h=1/M$ is the spatial mesh size.

\subsection{Experimental order of convergence}\label{eoc}
We test the numerical scheme proposed in Section~\ref{m-ffv} on a simple example in which we choose $\beta=\kappa, \gamma=0, \mb{F}=0$ and $G=-\partial_s(v\varrho)+\varrho^3+g$,
\begin{equation}
\label{eq:eoc}
\begin{split}
 &\partial_t \mb{X} = \partial^2_s\mb{X}, \\
 & \partial_t \varrho - \kappa^2\varrho = \partial_s^2\varrho -  \partial_s(v \varrho) + \varrho^3 + g, 
\end{split}
\end{equation}
where $v=-10$ and the function $g(u,t)$ is chosen such that the prescribed function $\varrho(u,t) = \cos(t)(\sin(2 \pi u) + \sin(4 \pi u))$ is the analytical solution of (\ref{eq:eoc}) satisfying the initial condition $\mb{X}(u,0)=(\cos 2\pi u, \sin 2\pi u,  0)^T, \varrho(u,0)= \sin(2 \pi u) + \sin(4 \pi u)$. Here $t\in[0,T]$ where $T=45$. 
Assuming that error estimates $err(M) = const (1/M)^{EOC}$ depend on the number of finite volumes $M$, the value of the experimental order of convergence (EOC) between two levels of meshes containing $M_1$ and $M_2$ finite
volumes is given by
\[
EOC = \frac{\log( err(M_1)/err(M_2))}{\log(M_2/M_1)}.
\]
The computational results are summarized in Table~\ref{tab:eoc}. They indicate the second order of convergence of the proposed numerical scheme. 

\begin{table}[ht]

\begin{center}
    
\caption{\small The experimental order of convergence for errors measured in the discrete $L^1$ and $L^\infty$ norms.}

\label{tab:eoc}

\scriptsize
\begin{tabular}{p{0.6cm}|p{2cm}|p{1.5cm}|p{2cm}|p{1.5cm}}
\hline
\multicolumn{1}{c|}{}    & \multicolumn{2}{c|}{$L^1((0,T); L^\infty(S^1))$}    &   \multicolumn{2}{c}{$L^\infty((0,T); L^\infty(S^1))$}  \\
\hline
$M$ &   error norm &  EOC  & error norm   &  EOC    \\
\hline
100   & 1.2975$\cdot 10^{-3}$ & -- & 2.8703$\cdot 10^{-3}$ & -- \\
200   & 3.3382$\cdot 10^{-4}$ & 1.9586 & 7.4222$\cdot 10^{-4}$ &  1.9513 \\
300   & 1.5312$\cdot 10^{-4}$ & 1.9222 & 3.4151$\cdot 10^{-4}$ &  1.9145 \\
400   & 8.9079$\cdot 10^{-5}$ & 1.8830 & 1.9907$\cdot 10^{-4}$  & 1.8760 \\
500   & 5.9055$\cdot 10^{-5}$ & 1.8421 & 1.3213$\cdot 10^{-4}$ &  1.8368 \\
\hline
\end{tabular}

\end{center}

\end{table}

\section{Numerical examples}\label{sec-numescheme}
In this section we present two examples of the evolution of closed curves in 3D. In the first example, we consider a simple curve without knots. On the other hand, in the second example we investigate the evolution of initial knotted curves. 

\subsection{An example of an evolution of 3D simple unknotted curves and scalar quantity}\label{sec:simpleex}
In the first numerical example, we consider a system of governing equations where the normal and binormal velocities and the external source term are given by $\beta = \kappa, \gamma=0$ and $\mb{F}=0$. Since $\beta\mb{N}=\kappa\mb{N}=\partial^2_s \mb{X}$, i.e.  
\begin{equation}
\label{example-system-num}
\begin{split}
 &\partial_t \mb{X} = \partial^2_s \mb{X},   \\
 & \partial_t \varrho - \kappa^2\varrho = \partial_s^2\varrho .
\end{split}
\end{equation}
A solution $(\mb{X},\varrho)$ is subject to the initial condition $\mb{X}(u,0) = (\cos 2\pi u, \sin 2 \pi, \sin 8 \pi u)^T$, $\varrho(u,0)=1$ if $u\in (1/4, 3/4)$ and $\varrho(u,0)=0$, otherwise. The evolution of the simple unknotted curve is shown in Fig.~\ref{fig4}. The values of the scalar quantity $\varrho$ are displayed by the color function. In the subsequent figures shown in Fig.~\ref{fig4} we depict the evolution of the initial curve approaching the shrinking circle for time levels $t\in \{0.062, 0.124, 0.19, 0.25\}$.

\begin{figure}
\begin{center}
\includegraphics[width=0.3\textwidth]{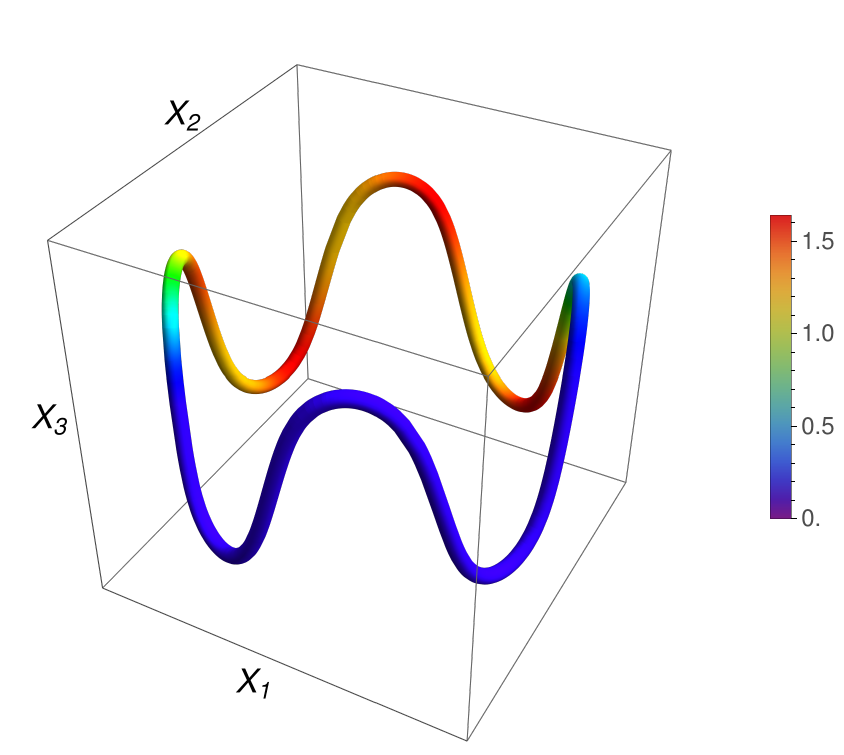}
\qquad\qquad
\includegraphics[width=0.3\textwidth]{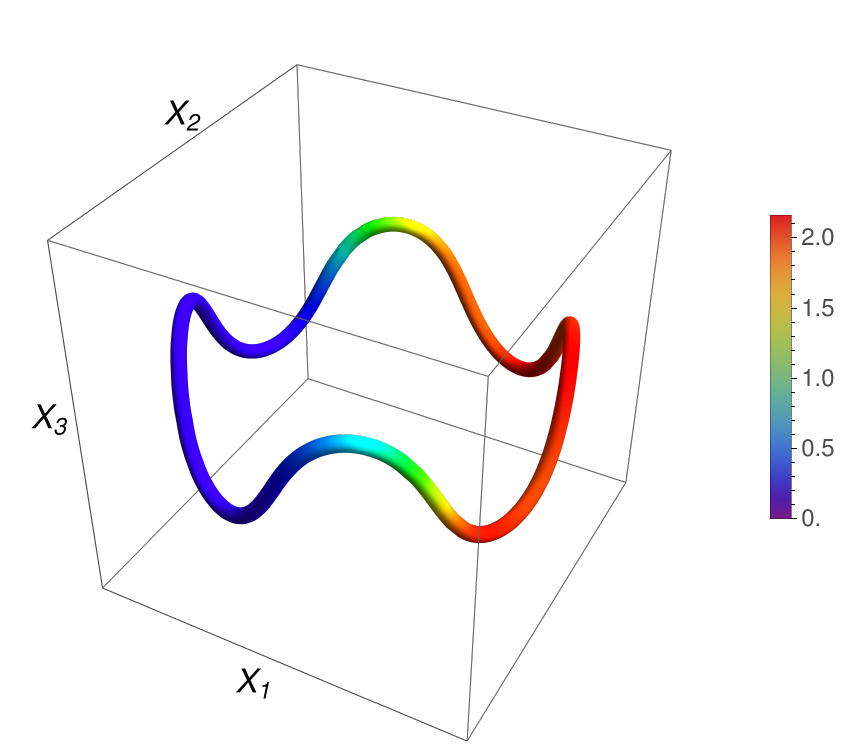}

\centerline{\scriptsize  $t = 0.062$ \hspace{0.35\textwidth} $t = 0.124$}

\medskip

\includegraphics[width=0.3\textwidth]{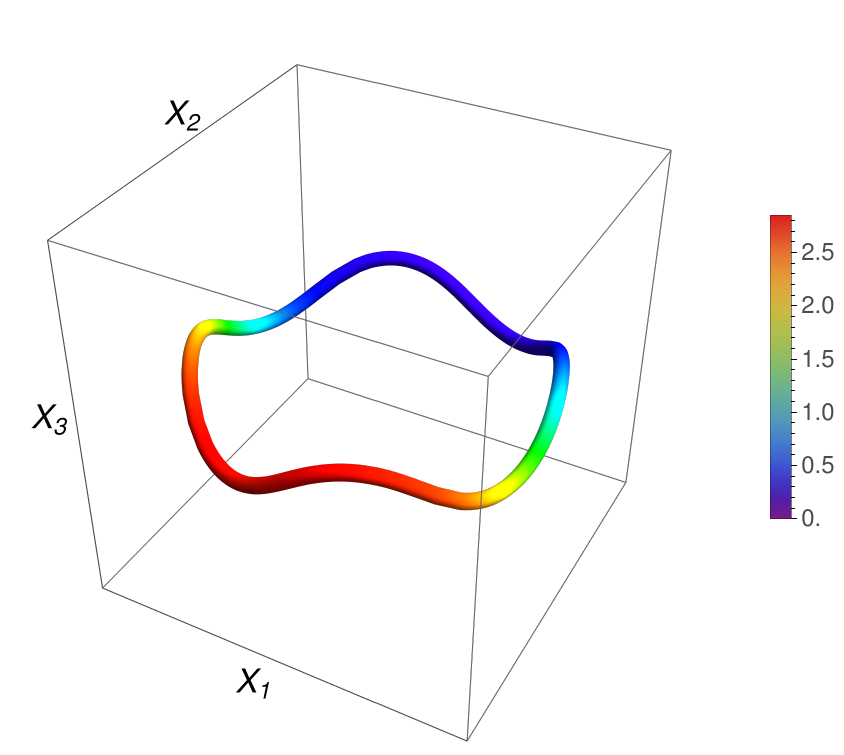}
\qquad\qquad
\includegraphics[width=0.3\textwidth]{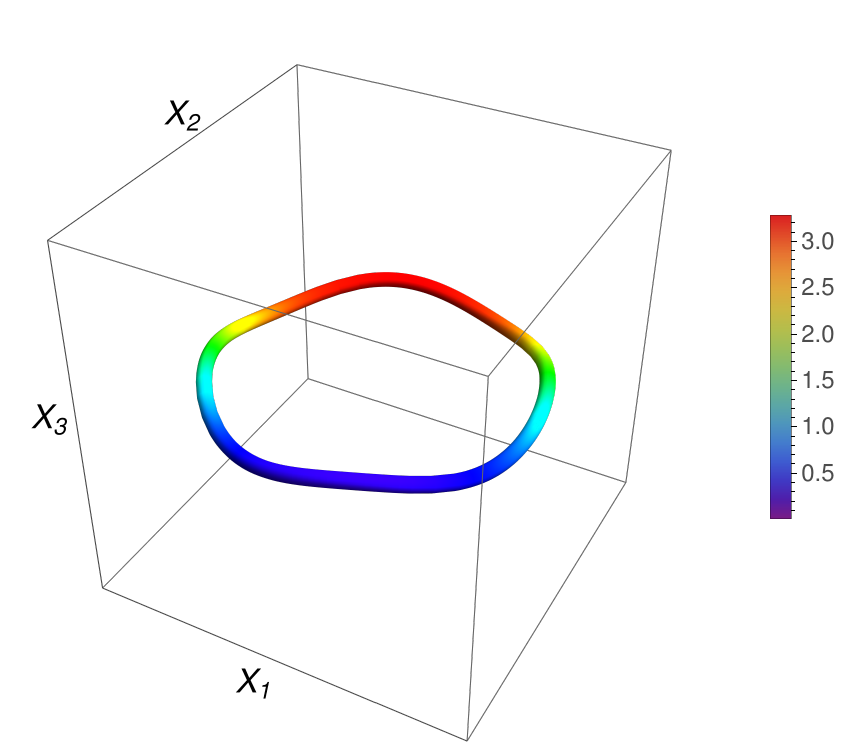}

\centerline{\scriptsize  $t = 0.19$ \hspace{0.35\textwidth} $t = 0.25$}

\caption{Numerical solution of problem (\ref{eq:balance_differential}) with advection velocity $v=20$ and $F=0$. Evolving curves approach the planar curve shortening the flow shrinking to a point in finite time.}

\label{fig4}
\end{center}
\end{figure}

\subsection{An example of evolution of knotted curves and scalar quantity}\label{sec-knotted}

A Fourier knot is a closed curve in 3D that can be parameterized by a finite Fourier series in the parameter $u\in[0,1]$. As an example, we consider a figure-eight knot (also called Listing's knot) which can be parameterized by the following finite trigonometric series:
\begin{equation}
\mb{X}(u,0) = (\cos(4 \pi u), \ 
\sin(6 \pi u + 1/2), \ 
( \cos(10 \pi u + 1/2) + \sin(6 \pi u + 1/2)) /2  )^T,
\label{init-curve}
\end{equation}
$u\in[0,1]$. Its top-view shape is shown in Fig.~\ref{fig-fourier} (left). 
\begin{figure}
\begin{center}
\includegraphics[width=0.2\textwidth]{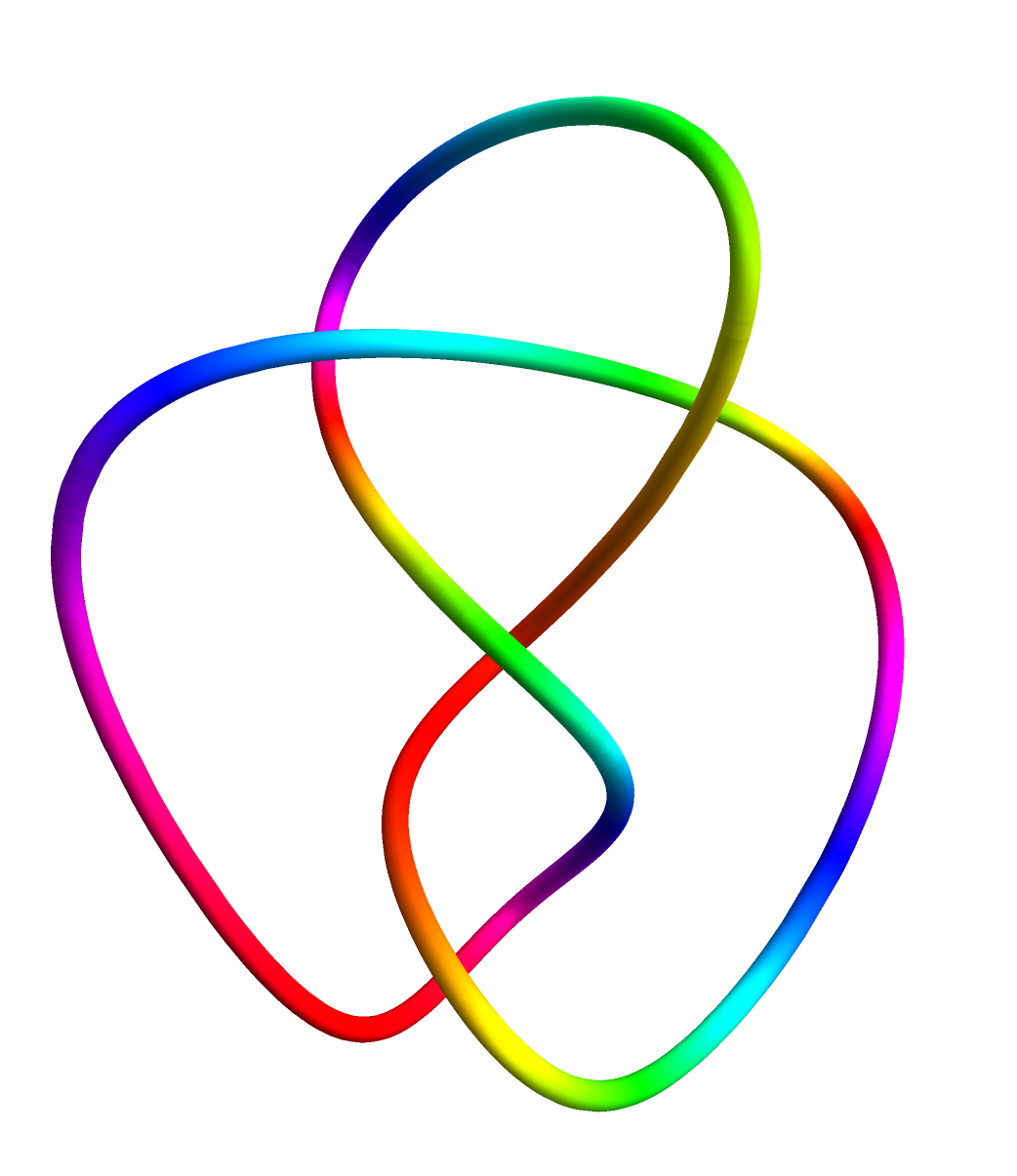}
\qquad
\includegraphics[width=0.2\textwidth]{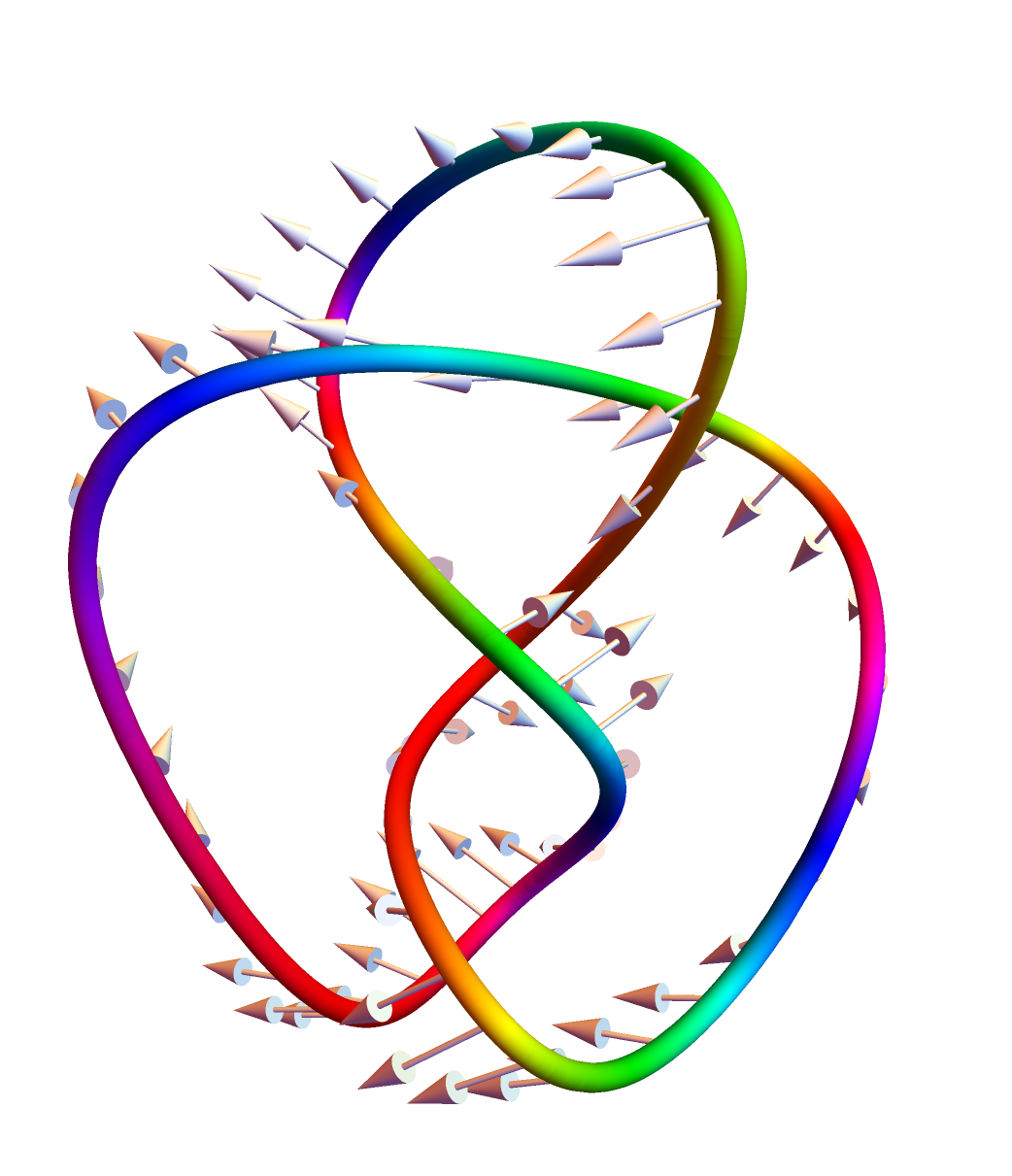}
\caption{Left: The initial Fourier knotted curve (left). Right: The arrows placed on the curve represent the Biot-Savart vector field $\mb{F}(\mb{X}, \Gamma_t), \mb{X}\in\Gamma_t$ given by (\ref{biot-savart-force}). }
\label{fig-fourier}

\end{center}
\end{figure}
As an example of a non-local source term $\mb{F}=\mb{F}_0$ we assume the external force corresponding to the Biot-Savart law. It represents the integrated influence of all points $\mb{X}$ that belong to the curve $\Gamma_t=\{ \mb{X}(s), s\in[0,L(\Gamma_t)]\}$ at a given point $\mb{X}\in\mathbb{R}^3$.
\begin{equation}
\mb{F}(\mb{X}, \Gamma_t) = \int_{\Gamma}
\frac{(\mb{X}-\mb{X}(s))\times \mb{T}(s)}{|\mb{X}-\mb{X}(s)|_\delta^3} ds .
\label{biot-savart-force}
\end{equation}
Following \cite{BKS2022}, we replaced the Euclidean distance between points $\mb{X}$ and $\mb{X}(s)$ by its regularization  $|\mb{X}-\mb{X}(s)|_\delta = (\delta^2 + |\mb{X}-\mb{X}(s)|^2)^\frac12$ where  $0<\delta\ll 1$ is a small regularization parameter. The vector field $\mb{F}$ corresponding to the regularized Biot-Savart force (\ref{biot-savart-force}) is shown in Fig.~\ref{fig-fourier} (right).

\begin{remark}\label{biot-savart}
Assume that the curve $\Gamma_t$ has zero torsion. Then the curve $\Gamma_t$ belongs to a plane. Let $\mb{X}\in\Gamma_t$. If $\mb{N}$ and  $\mb{T}$  are the unit normal and tangent vectors at $\mb{X}$, then the Biot-Savart force projected onto the normal and tangent vectors vanishes, that is, $\mb{F}(\mb{X}, \Gamma_t) \cdot \mb{N} = \mb{F}(\mb{X}, \Gamma_t) \cdot \mb{T} = 0$. In fact, if $\Gamma_t$ belongs to a plane, then the difference of the position vectors $\mb{X}-\mb{X}(s)$, the tangent vector $\mb{T}(s)$ in $\mb{X}(s)$ as well as the vectors $\mb{N}, \mb{T}$  belong to the same plane, and so $((\mb{X}-\mb{X}(s))\times \mb{T}(s))\cdot \mb{N} = ((\mb{X}-\mb{X}(s))\times \mb{T}(s))\cdot \mb{T} = 0$ for each $s\in[0,L(\Gamma_t)]$. 

For example, if the curve $\Gamma_t$ is a circle parameterized by $\mb{X}(s) = (\cos s, \sin s, 0)^T, s\in[0,2\pi]$, then it is easy to verify that $\mb{F}(\mb{X}, \Gamma_t) \cdot \mb{B} = \int_0^{2\pi} (1-\cos s) (\delta^2 + 2 (1-\cos s))^{-\frac32}$. Notice that $\mb{F}(\mb{X}, \Gamma_t) \cdot \mb{B} = \infty$ if $\delta=0$.
\end{remark}

\begin{remark}\label{biot-savart-linking}

If $\Gamma_t$ and $\tilde\Gamma_t$ are two non-intersecting curves in 3D then the Gauss linking integral of $\Gamma_t$ and $\tilde\Gamma_t$ can be defined as follows:
\begin{eqnarray*}
link(\Gamma_t,\tilde\Gamma_t ) &=& \frac {1}{4\pi }\oint_{\Gamma}\oint_{\tilde\Gamma}
\frac{det\left(\partial_s \mb{X}(s), \partial_{\tilde s} \tilde{\mb{X}}({\tilde s}), \mb{X}(s) - \tilde{\mb{X}}({\tilde s}) \right)}{|\mb{X}(s) - \mb{X}(\tilde{s})|^3} ds d\tilde{s}
\\
&=& 
- \frac {1}{4\pi }\oint_{\Gamma} \mb{F_0}(\mb{X}(s), \tilde{\Gamma_t}) \cdot \partial_s \mb{X}(s)  ds 
\end{eqnarray*}
where $\Gamma_t$ and $\tilde\Gamma_t$ are parameterized by $\mb{X}(s)$ and $\tilde{\mb{X}}$, respectively. The linking number $link(\Gamma_t,\tilde\Gamma_t )$ computes the total signed area of the image of the Gauss map divided by the area $4\pi$ of the unit sphere in 3D.
\end{remark}

\begin{figure}
\begin{center}
\includegraphics[width=0.25\textwidth]{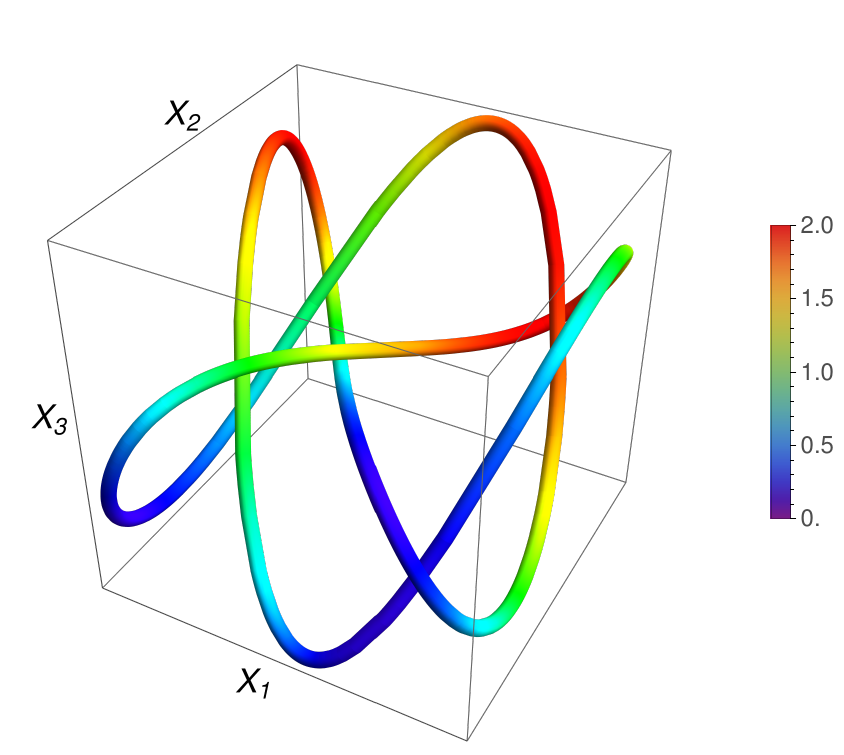}
\qquad \includegraphics[width=0.3 \textwidth]{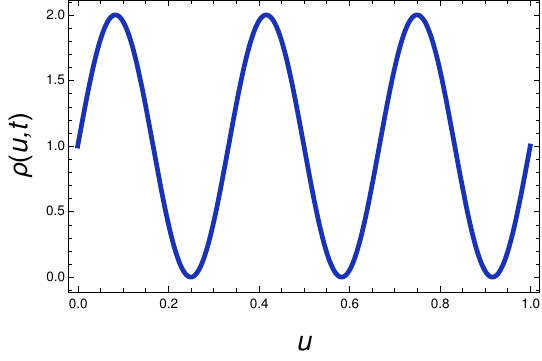}
\end{center}

\caption{The initial knotted Fourier curve $\mb{X}=(X_1, X_2, X_3)^T$ and the initial scalar quantity $\varrho(u, 0) = 1 +\sin(6\pi u)$.}
\label{fig-init}

\end{figure}

We consider the system of governing equation of the form:
\begin{equation}
\label{example-system-biot}
\begin{split}
 &\partial_t \mb{X} = \beta \mb{N}  + \gamma \mb{B}  +\mb{F}(\mb{X},\Gamma_t), \\
 & \partial_t \varrho - \kappa\beta\varrho = \partial_s^2\varrho .
\end{split}
\end{equation}
where $\beta = \kappa - \frac{2\pi}{L(\Gamma_t)}$. As an initial condition for the system (\ref{example-system-biot}) we consider the Fourier knotted curve shown in Fig.~\ref{fig-init} and the initial scalar quantity $\varrho(u,0) = 1 +\sin(6 \pi u)$.

\begin{figure}
\begin{center}
\includegraphics[width=0.33\textwidth]{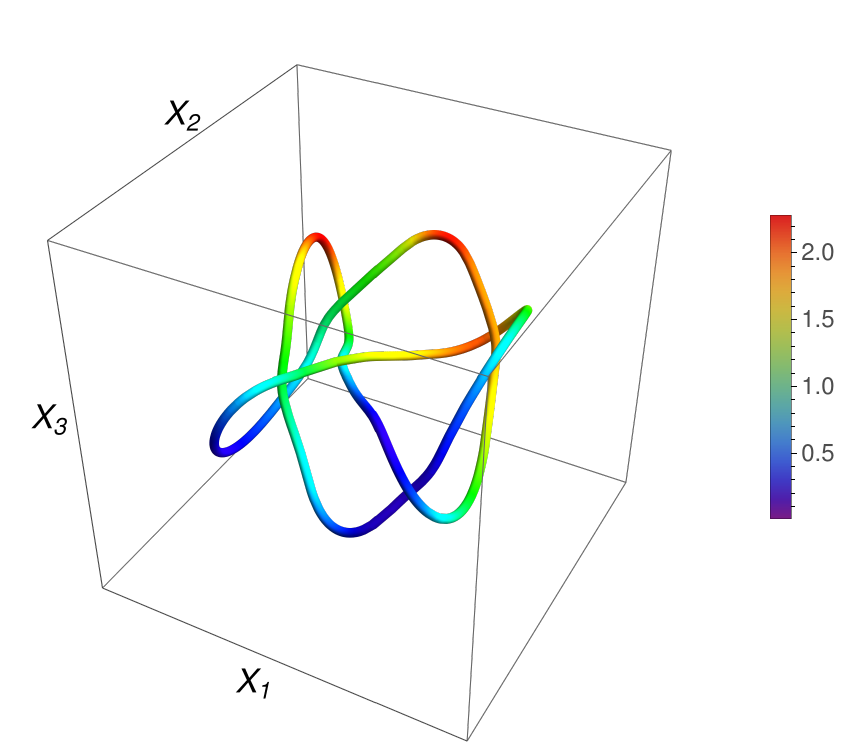}
\qquad \includegraphics[width=0.3\textwidth]{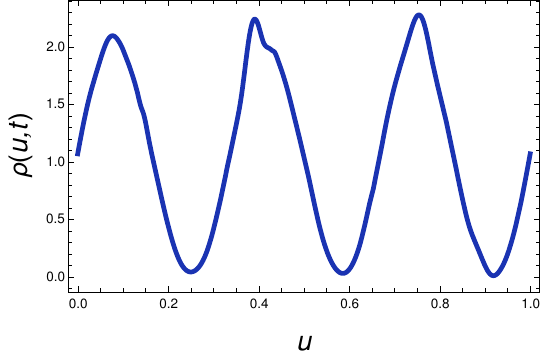}
\centerline{\scriptsize $t = 0.02$}

\includegraphics[width=0.33\textwidth]{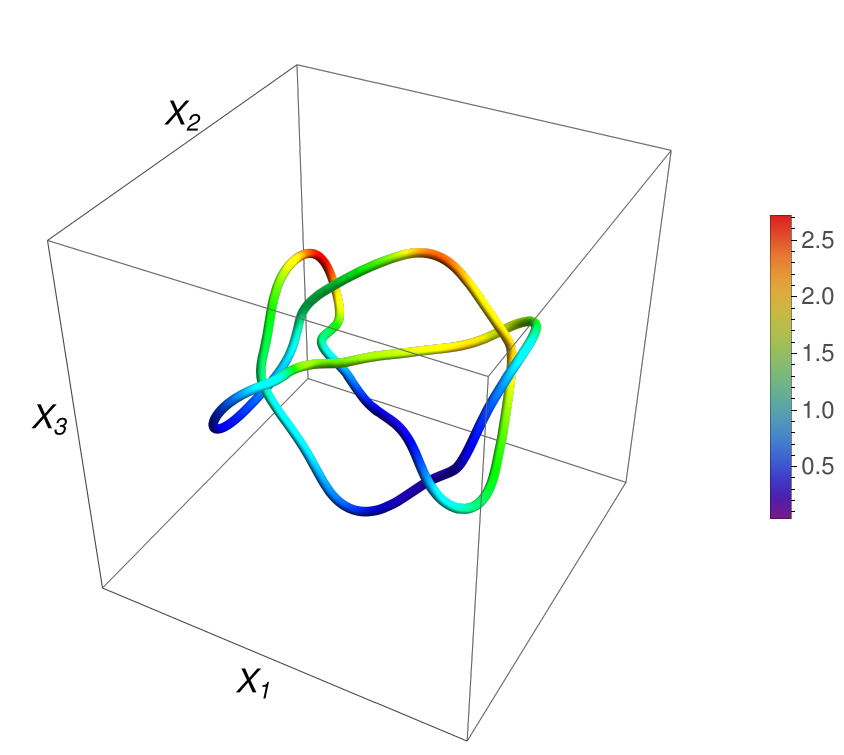}
\qquad \includegraphics[width=0.3\textwidth]{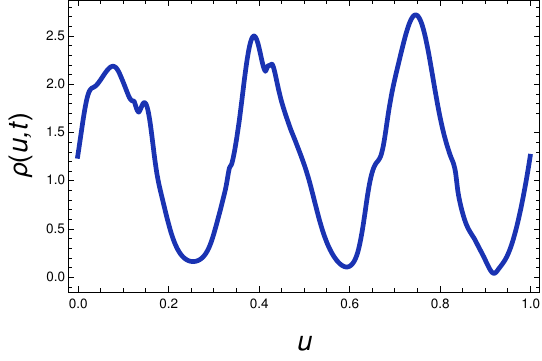}
\centerline{\scriptsize $t = 0.063$}

\includegraphics[width=0.33\textwidth]{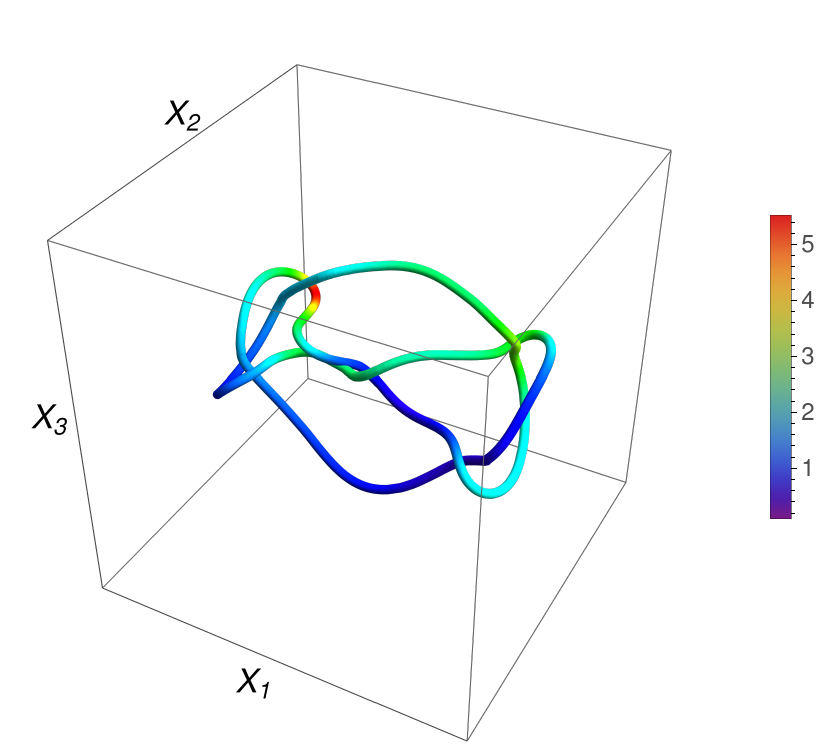}
\qquad \includegraphics[width=0.3\textwidth]{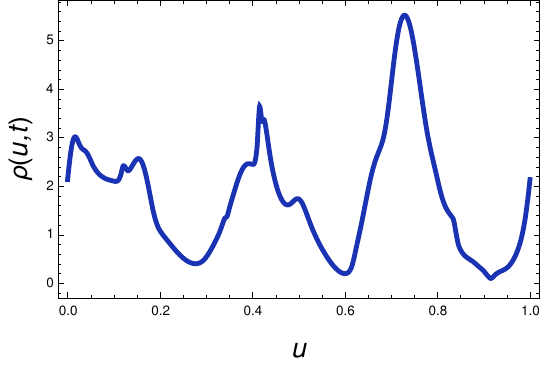} 
\centerline{\scriptsize $t = 0.124$}

\includegraphics[width=0.33\textwidth]{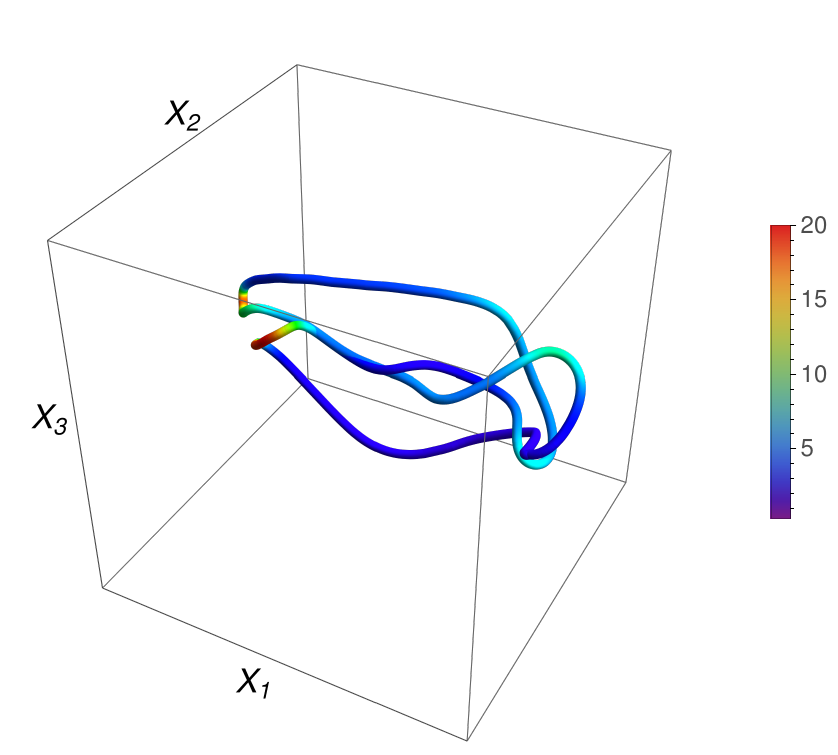}
\qquad 
\includegraphics[width=0.3\textwidth]{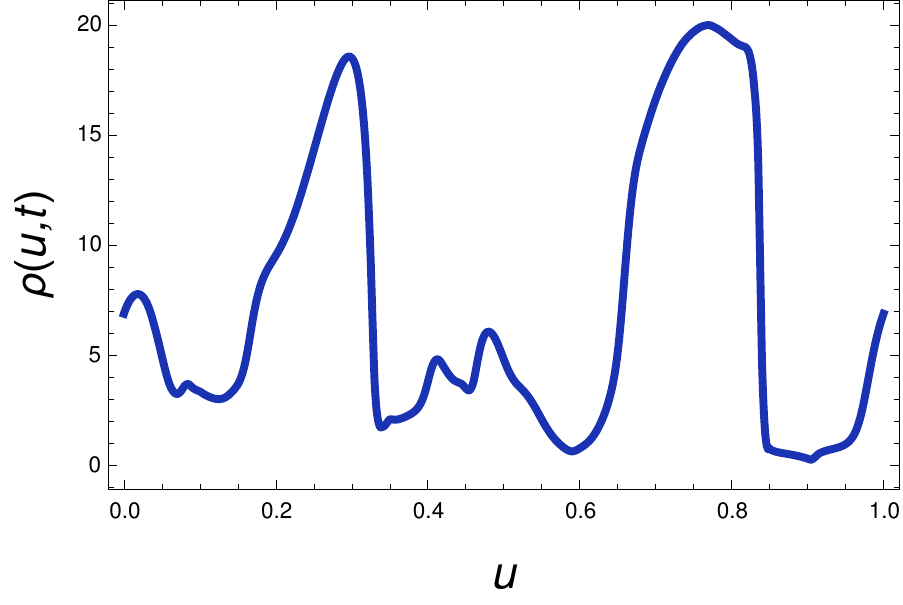} 
\centerline{\scriptsize $t = 0.237$}

\caption{Evolution of the initial knotted curve shown in Fig.~\ref{fig-init} with the velocity $\gamma=\varrho$ and the external force given by the regularized Biot-Savart force given by (\ref{biot-savart-force}).}

\label{fig-knot-BS}
\end{center}
\end{figure}

In the first experiment shown in Fig.~\ref{fig-knot-BS} we investigate the motion of the initial knotted curve (\ref{fig-init}) with the binormal velocity $\gamma=\varrho$ linearly depending on the scalar quantity $\varrho$. The external force given by the regularized Biot-Savart force (\ref{biot-savart-force}). 
The presence of the Biot-Savart force $\mb{F}$ ensures that the evolved curves preserve the topological structure of a knotted curve.

\begin{figure}
\begin{center}
\includegraphics[width=0.33\textwidth]{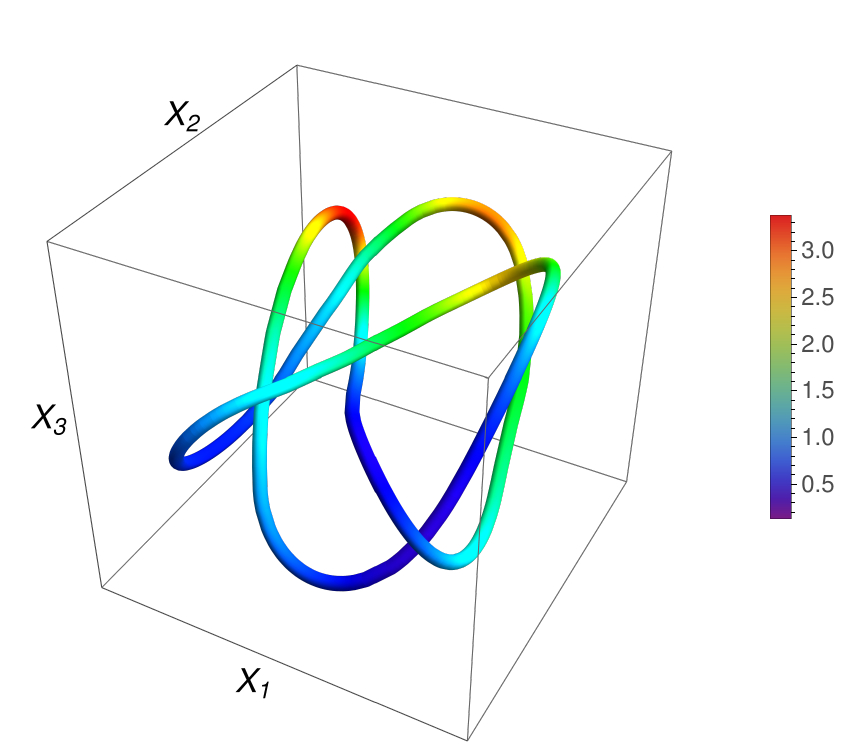}
\qquad \includegraphics[width=0.3\textwidth]{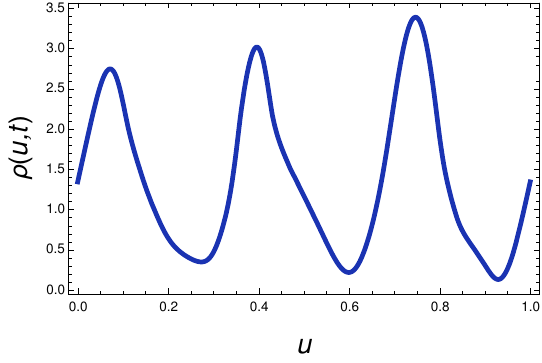}
\centerline{\scriptsize $t = 0.1225$}

\includegraphics[width=0.33\textwidth]{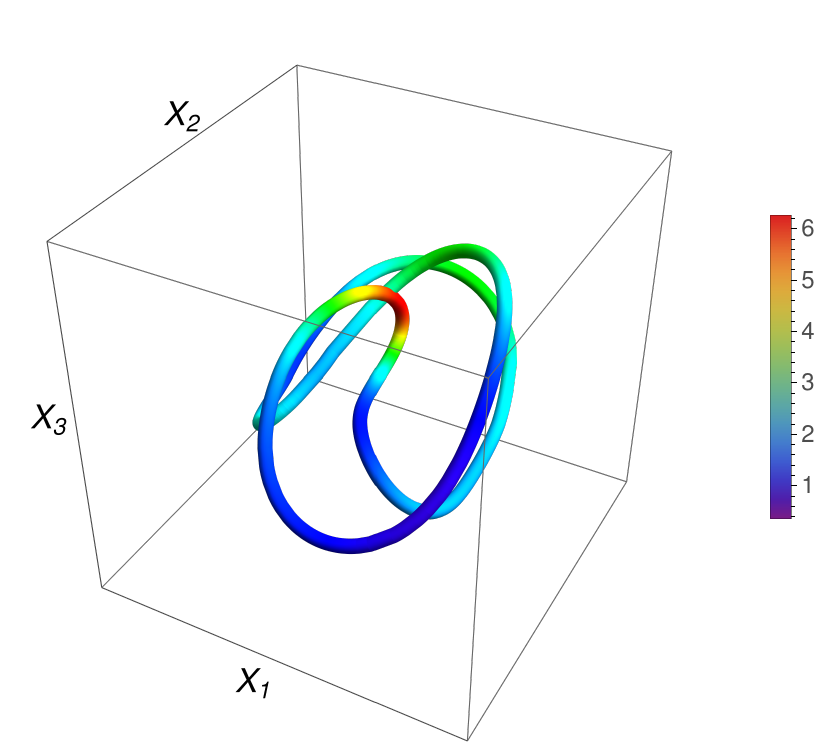}
\qquad \includegraphics[width=0.3\textwidth]{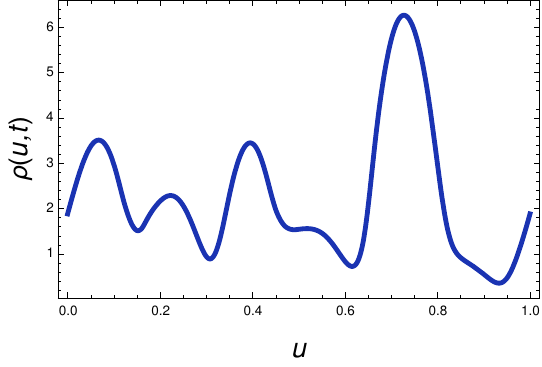}
\centerline{\scriptsize $t = 0.25$}

\includegraphics[width=0.33\textwidth]{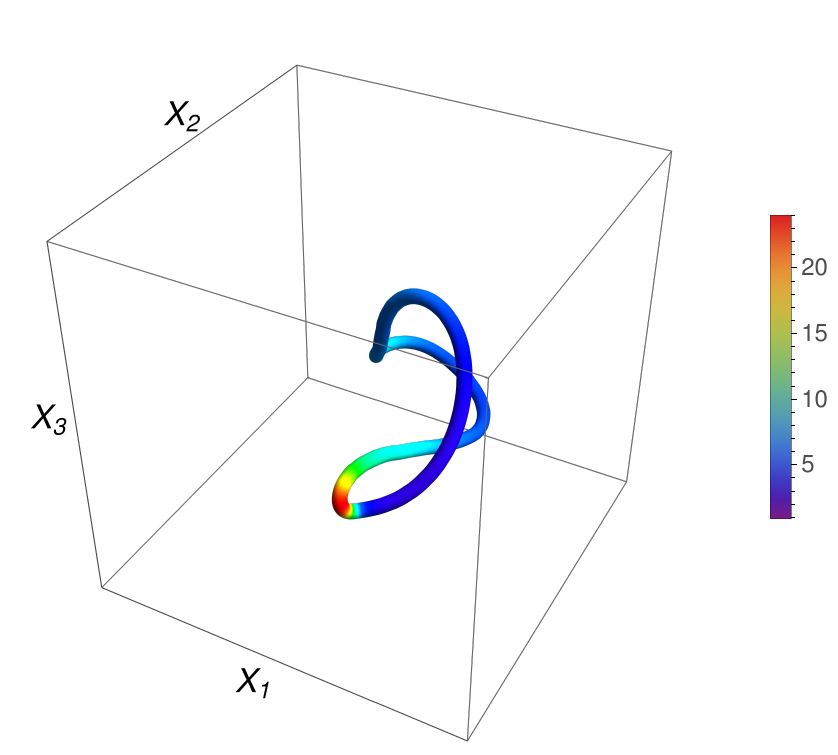}
\qquad \includegraphics[width=0.3\textwidth]{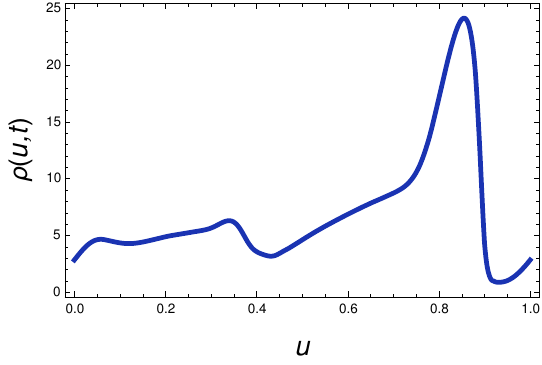}
\centerline{\scriptsize $t = 0.4$}

\caption{Evolution of the initial knotted curve with the binormal velocity $\gamma=\varrho$ and the vanishing force $\mb{F}=0$.}
\label{fig-knot}

\end{center}
\end{figure}

The third experiment is shown in Fig.~\ref{fig-knot}. The same initial knotted curve in Fig.~\ref{fig-init} is evolved by the binormal velocity $\gamma=\varrho$ and vanishing external force $\mb{F}=0$. In contrast to the experiments shown in Fig.~\ref{fig-knot-BS} the knotted structure is not preserved and the evolved curve approaches a circular curve in the plane. 

\section{Conclusions}
In this paper, we investigated the motion of closed smooth curves that evolve in $\mathbb{R}^3$ coupled with the evolution of a scalar quantity evaluated over the evolving curve. We proved the local existence and uniqueness of classical H\"older smooth solutions. The proof is based on the analysis of the position vector equation and the parabolic equation for the evolving scalar quantity. We also proposed a numerical discretization scheme for numerical approximation of the evolving family of curves and scalar quantity. We presented several numerical examples of evolving curves and the scalar quantity. 
\bigskip

\noindent{\bf Acknowledgments} 

\noindent The first and second authors received partial support from the project 25-18265S of the Czech Science Foundation.  The third author received support from the Slovak Research Agency  VEGA 1-0493-24.

\bigskip
\noindent{\bf Conflicts of Interest.}

\noindent This work does not have any conflicts of interest

\end{document}